\documentclass[reqno]{amsart}
\usepackage[utf8]{inputenc}
\usepackage[T1]{fontenc}
\usepackage{amsfonts}
\usepackage[all]{xy}
\usepackage[title,titletoc]{appendix}
\usepackage{amsmath}
\usepackage{amssymb}
\usepackage{amsthm}
\usepackage{mathtools}
\usepackage{comment}
\usepackage{esint}
\usepackage{hyperref}
\usepackage{xcolor}
\usepackage{tikz-cd}
\hypersetup{
   colorlinks,
   linkcolor={red!80!black},
   citecolor={blue!70!black},
   urlcolor={blue!90!black}
}
\usepackage{microtype}
\usepackage{dutchcal}
\usepackage{upgreek}

\usepackage{tikz}
\usepackage{fourier}
\usepackage[mathscr]{euscript}
\usepackage[maxbibnames=99,backend=biber, sorting=nyt, doi=false, url=false]{biblatex}
\usepackage{pgfornament}

\newcommand{\R}{\mathbb{R}}
\newcommand{\C}{\mathbb{C}}

\newcommand{\N}{\mathbb{N}}
\newcommand{\from}{\colon}

\DeclareMathOperator{\Lip}{Lip}
\DeclareMathOperator{\id}{id}

\DeclareMathOperator{\supp}{supp}
\DeclareMathOperator{\sgn}{sgn}
\newcommand{\ud}{\,\mathrm{d}}

\newcommand{\Hom}{\mathrm{Hom}}

\DeclareMathOperator*{\ulim}{\overline{\lim}}
\DeclareMathOperator{\clos}{clos}
\DeclareMathOperator{\im}{im}

\newtheorem{thm}{Theorem}[section]
\newtheorem{theorem}[thm]{Theorem}

\newtheorem{lemma}[thm]{Lemma}
\newtheorem{prop}[thm]{Proposition}
\newtheorem{cor}[thm]{Corollary}
\newtheorem{defn}[thm]{Definition}
\theoremstyle{remark}
\newtheorem{remark}[thm]{Remark}

\title{Ergodic maps and the cohomology of nilpotent Lie groups}

\author{Gioacchino Antonelli}
\address[Gioacchino Antonelli]
{New York University, Courant Institute of Mathematical Sciences, 251 Mercer Street, 10012, New York, USA}
\email{ga2434@nyu.edu}
\author{Robert Young}
\address[Robert Young]
{New York University, Courant Institute of Mathematical Sciences, 251 Mercer Street, 10012, New York, USA}
\email{ryoung@cims.nyu.edu}

\begin{document}

\begin{abstract}
  In this paper, we study how the cohomology of nilpotent groups is affected by Lipschitz maps. We show that, given a smooth Lipschitz map $f$ between two simply-connected nilpotent Lie groups $G$ and $H$, there is a map $\psi$ that induces an ergodic measure on the space of functions from $G$ to $H$. We call such maps \emph{ergodic maps}.

  We show that when $\psi$ is an ergodic map, the pullback $\psi^*\omega$ of a differential form $\omega$ admits a well-defined \textit{amenable average} $\overline{\psi^{*}}\omega$, and $\overline{\psi^*}$ is a homomorphism of cohomology algebras. In the case that $f$ is a quasi-isometry, the ergodic map $\psi$ is also a quasi-isometry, and $\overline{\psi^*}$ is an isomorphism. This lets us generalize and provide a simplified, self-contained proof of the theorem due to Shalom, Sauer, and Gotfredsen--Kyed that quasi-isometric nilpotent groups have isomorphic cohomology algebras.
\end{abstract}

\maketitle

\section{Introduction}

The question of when two nilpotent groups are quasi-isometric has motivated many developments in geometric group theory. Since any finitely generated nilpotent group is a cocompact lattice in a nilpotent Lie group, it suffices to consider nilpotent Lie groups. Clearly, two isomorphic nilpotent Lie groups are quasi-isometric, but it remains open whether two quasi-isometric groups are isomorphic.

One of the first results in this direction was Pansu's theorem \cite{PansuCroissance, Pansu}, which says that the asymptotic cone of a nilpotent group is isomorphic to its associated graded group, and that two quasi-isometric nilpotent groups must have isomorphic asymptotic cones. Shalom \cite{Shalom} was the first to exhibit two nilpotent groups with the same associated graded group which are not quasi-isometric, by showing that the Betti numbers of finitely-generated nilpotent groups are invariant under quasi-isometry. Sauer \cite{Sauer} extended this result, showing that the cup product structure is likewise preserved under quasi-isometry, and Gotfredsen and Kyed extended this to arbitrary nilpotent Lie groups \cite{GotfredsenKyed}.
More recently, Llosa Isenrich, Pallier, and Tessera \cite{LlosaIsenrichPallierTessera} gave an example of two groups with isomorphic asymptotic cones which can be distinguished using their Dehn functions.

Shalom and Sauer's results are particularly promising because they point to a possible strategy for solving the conjecture. The cup product structure is the first of a series of higher cohomology operations that operate on two, three, or more cohomology classes at a time. If these higher operations are likewise invariant under quasi-isometry, in the sense that a quasi-isometry from $G$ to $H$ induces a $C_\infty$--morphism from $H^*(H)$ to $H^*(G)$, that would imply that the minimal models of $H^*(H)$ and $H^*(G)$ are isomorphic \cite[Theorem 9.1]{KadeCohom}, and in turn, that $G$ and $H$ are isomorphic.

In this paper, we generalize the result of Shalom, Sauer, and Gotfredsen--Kyed to maps that are not quasi-isometries and give a new, more geometric proof of their results based on pullbacks of differential forms. Before stating our main theorems, we introduce some notation. Let $\phi\from G \to H$ be a Lipschitz map between two amenable connected Riemannian Lie groups, and let $\omega$ be a left-invariant differential form with complex coefficients. Then the pullback $\phi^*\omega$ is a (possibly non-smooth) bounded differential form on $G$, and since $G$ is amenable, one can average $\phi^*\omega$ to obtain a left-invariant differential form.

That is, let $L^\infty(G)$ be the set of complex-valued $L^\infty$ functions on $G$, let $\mathfrak{g}$ be the Lie algebra of $G$, and let $M\from L^\infty(G)\to \C$ be a left-invariant mean. For every left-invariant $d$--vector field $\lambda \in \wedge^d \mathfrak{g}$ on $G$, we can evaluate $\phi^*\omega$ on $\lambda$ to obtain a bounded function $\phi^*\omega(\lambda)$ on $G$, and there is a unique left-invariant form $\overline{\phi^*}\omega \in C^d(\mathfrak{g})$ such that 
\begin{equation}\label{eqn:DefinedAs}
  \overline{\phi^*}\omega(\lambda) := M(\phi^*\omega(\lambda)), \qquad \text{ for all }\lambda\in \wedge^d \mathfrak{g}.
\end{equation}
We denote this by $\overline{\phi^*}\omega$ because it is the result of averaging the pullback $\phi^*\omega$.

It is hard to characterize $\overline{\phi^*}$ in the general case, because a Lipschitz map from $G$ to $H$ can have different behavior at different locations and scales. In this paper, however, we will show two main results. First, if $\phi$ is an ergodic map (i.e., it satisfies an ergodicity condition that we will define below), then $\overline{\phi^*}$ is a homomorphism of cohomology algebras. Second, for any map $f$ from $G$ to $H$, there is a limit of translates of $f$ which is an ergodic map. In particular, if $f$ is a quasi-isometry from $G$ to $H$, then there is an ergodic quasi-isometry $\phi$, and $\overline{\phi^*}$ is an isomorphism. This lets us recover the theorems of Shalom, Sauer, and Gotfredsen--Kyed and extend their results from quasi-isometries to the broader class of Lipschitz maps between nilpotent groups.

Before we state our results, we first define ergodic maps (see Section~\ref{sec:ergodic} for full details). Let $Y_0$ be the space of smooth maps $f$ from $G$ to $H$ such that all derivatives of $f$ are bounded, and $f(\id_G)=\id_H$. We can define a right action of $G$ on $Y_0$ by letting $(f\cdot g)(x) := f(g)^{-1} f(g\cdot x)$. Since $(f\cdot g)(\id_G)=\id_H$ and $f\cdot g$ has bounded derivatives, the closure of the orbit $f\cdot G$ is compact by the Arzelà--Ascoli theorem. 

For any $n>0$, let $B_n\subset G$ be the ball of radius $n$ centered at $\id_G$. Given $f\in Y_0$, there is a measure $\mu_n$ on $Y_0$ such that for any continuous function with compact support $A \from Y_0\to \R$,
$$
\int A \ud \mu_n := \fint_{B_n} A(f\cdot g)\ud g,
$$
i.e., $\int A \ud \mu_n$ is the expected value of $A(f\cdot g)$ for $g$ drawn uniformly from $B_n$, see Section~\ref{sec:bounded-derivs} for notation.
We say that $f$ is an \emph{ergodic map} if the measures $\mu_n$ converge weakly to some limit measure $\mu$, and $\mu$ is an ergodic measure on $Y_0$ with respect to the right action of $G$ defined above. 

For example, the map $f_1\from \R\to \R^2$, $f_1(x) = (x, \sin x)$ is ergodic, with limit measure $\mu$ supported on maps of the form 
$$
h_\theta(x) = (x, \sin(x+\theta)-\sin \theta),
$$
but $f_2(x) = (x, |x|)$ is not; its limit measure is a two-point measure supported on $x\mapsto (x, \pm x)$.

Our main result states that any ergodic map induces an algebra homomorphism between the real cohomology algebras $H^*(\mathfrak{h};\mathbb R)$ and $H^*(\mathfrak{g};\mathbb R)$. 

\begin{thm}\label{thm:induced-map}
  Let \(G\) and \(H\) be simply connected nilpotent Lie groups, and let $\psi\in Y_0$ be an ergodic map. For any $d\geq 0$, any real-valued left-invariant form $\omega\in C^d(\mathfrak{h};\mathbb R)$, and any left-invariant $d$--vector field $\lambda\in \wedge^d\mathfrak{g}$, let
  \begin{equation}\label{eqn:ContinuousInvariant}
  \overline{\psi^*}\omega(\lambda) := \lim_{n\to \infty} \fint_{B_n} \psi^*\omega(\lambda) \ud g.
  \end{equation}
  Then this limit exists, and $\overline{\psi^*}$ induces a homomorphism $\overline{\psi^*}\from H^*(\mathfrak{h};\mathbb R)\to H^*(\mathfrak{g};\mathbb R)$ of cohomology algebras over $\mathbb R$. If $\psi$ is a quasi-isometry, then $\overline{\psi^*}$ is an isomorphism.
\end{thm}

Furthermore, ergodic maps are common, as shown by the following result.
\begin{prop}\label{prop:ExistenceOfErgodic}
  Let $G$ and $H$ be simply connected nilpotent Lie groups, and let $f\in Y_0$. Then there is a sequence $g_n\in G$ such that $f\cdot g_n$ converges to a map $\phi\in Y_0$ uniformly on compact subsets, and $\phi$ is ergodic.
  In particular, if $f$ is a quasi-isometry or a quasi-isometric embedding, then so is $\phi$.
\end{prop}

As a consequence of Theorem~\ref{thm:induced-map} and Proposition~\ref{prop:ExistenceOfErgodic}, we prove the following theorem.
\begin{thm}\label{thm:mainThm}
  Let $G$ and $H$ be quasi-isometric simply connected nilpotent Lie groups. Then there are a smooth Lipschitz quasi-isometry $\phi\from G\to H$ and a real-valued left-invariant mean $M\from L^\infty(G;\mathbb{R})\to \mathbb R$ such that $\overline{\phi^*}\from H^*(\mathfrak{h};\mathbb R)\to H^*(\mathfrak{g};\mathbb R)$ defined as in \eqref{eqn:DefinedAs} is an isomorphism of cohomology algebras over $\mathbb R$.
\end{thm}
This implies the results of Shalom, Sauer, and Gotfredsen--Kyed and constructs an explicit isomorphism from $H^*(\mathfrak{h};\mathbb R)$ to $H^*(\mathfrak{g};\mathbb R)$. If one could show that this isomorphism also preserves higher cohomology operations, it would imply that $\mathfrak{g}$ and $\mathfrak{h}$ are isomorphic, but this remains open.

Since Theorem~\ref{thm:induced-map} is not limited to quasi-isometries, it can be applied to other maps. One example, which we learned from Kyle Hansen, is the case of maps with positive asymptotic degree (see \cite{BerdGuthManin}). If $G$ and $H$ are simply connected nilpotent groups of the same dimension, $\psi\from G\to H$ is Lipschitz, and $\omega$ is the volume form on $H$, we say that $\psi$ has \emph{positive asymptotic degree} if 
$$\limsup_{R\to \infty} \frac{\int_{B^G_R} \psi^*\omega}{\big|B^G_R\big|} > 0,$$
where $B^G_R$ is the ball of radius $R$ around the origin in $G$. If $\psi$ is an ergodic map with positive asymptotic degree, then $\overline{\psi^*}\omega\ne 0$, so by Poincar\'e duality (see the proof of Theorem~\ref{thm:induced-map}), $\overline{\psi^*}$ is injective on cohomology. Hence, for example, there is no ergodic Lipschitz map $\psi$ from the Heisenberg group to $\mathbb{R}^3$ with positive asymptotic degree.

\subsection*{Outline of paper}
In order to use results on cohomology with coefficients in a unitary representation, we use complex coefficients throughout this paper, so $C^*(\mathfrak{g}),H^*(\mathfrak{g})$ denote $C^*(\mathfrak{g};\C),H^*(\mathfrak{g};\C)$, respectively.

In Section \ref{sec:bounded-derivs} we start by defining the locally compact space $Y_0$ of smooth maps with bounded derivatives from $G$ to $H$, and an action of $G$ on $Y_0$. We prove Proposition~\ref{prop:ExistenceOfErgodic} in Section~\ref{sec:ergodic} by showing that if $\phi \in Y_0$, then its limit measure is a $G$--invariant probability measure. We then use the Krein--Milman theorem to construct an ergodic measure $\mu$, and use an ergodic theorem of E.\ Lindenstrauss (see Theorem~\ref{thm:lindenstrauss}) to construct an ergodic $\phi$.

For any left-invariant differential form $\omega\in C^d(\mathfrak{h})$ and any $\psi\in Y_0$, the pullback $\psi^*\omega$ is a smooth bounded $d$--form which depends on $\psi$. In Section~\ref{sec:Pullbacks} we define an algebra $C^{\infty,\infty}(Y_0)$ of complex-valued functions on $Y_0$ such that for each $\alpha\in C^{\infty,\infty}(Y_0)$ and $\psi\in Y_0$, there is a smooth function $\alpha_\psi\in C^{\infty}(G)$ with bounded derivatives.

    Using this algebra, we construct a cochain complex $C^*(\mathfrak{g};C^{\infty,\infty}(Y_0))$ that represents forms on $G$ that depend on an element of $Y_0$, i.e., for every $\beta\in C^d(\mathfrak{g};C^{\infty,\infty}(Y_0))$ and $\psi\in Y_0$, there is a smooth bounded $d$--form $\beta_\psi$ on $G$. This lets us construct a universal pullback map $T^\sharp \from C^*(\mathfrak{h})\to C^*(\mathfrak{g};C^{\infty,\infty}(Y_0))$ such that for each left-invariant form $\omega\in C^d(\mathfrak{h})$, and $\psi\in Y_0$,
$$
(T^\sharp\omega)_\psi = \psi^*\omega.
$$
Such $T^\sharp$ is a map of differential graded algebras, and thus induces an algebra homomorphism from the cohomology $H^*(\mathfrak{h})$ to $H^*(\mathfrak{g};C^{\infty,\infty}(Y_0))$, see Lemma \ref{lem:Tsharp}.

In Section~\ref{sec:VanishingLemma} we prove a consequence of a vanishing result of Blanc, see Lemma~\ref{lem:vanishing2}. We use this result to show that if $\mu$ is ergodic under the action of $G$, $L^2(\mu)$ is the set of $L^2$ functions on $Y_0$, and $M\from L^2(\mu)\to \mathbb{C}$ is the map $M(f):=\int_{Y_0} f\ud \mu$, then $M$ induces an algebra homomorphism from the cohomology ${H}^*(\mathfrak{g};C^{\infty,\infty}(Y_0))$ to $H^*(\mathfrak{g})$, see Lemma \ref{lem:McommutesCupProduct} in Section \ref{sec:Integration}. Since we can factor $\overline{\phi^*}=M\circ T^\sharp$ (see Lemma \ref{lemma:LimitExists}), this implies that $\overline{\phi^*}$ is an algebra homomorphism. We note that $\overline{\phi^*}$ descends to a map from $H^*(\mathfrak{h};\mathbb R)$ to $H^*(\mathfrak{g};\mathbb R)$, and thus this settles the first part of Theorem~\ref{thm:induced-map}.

Finally, in Section~\ref{sec:proof}, we prove the last parts of Theorem~\ref{thm:mainThm}, and Theorem~\ref{thm:induced-map}.
We prove that if $\phi$ is an ergodic quasi-isometry and $\omega$ is the volume form on $H$, then $\overline{\phi^*}(\omega)\ne 0$, see Corollary \ref{cor:volcontrol}. Thus, by Poincaré duality, $\overline{\phi^*}$ is an injective homomorphism from $H^*(\mathfrak{h};\mathbb R)$ to $H^*(\mathfrak{g};\mathbb R)$. By applying the same argument with $G$ and $H$ switched, we obtain an injective homomorphism from $H^*(\mathfrak{g};\mathbb R)$ to $H^*(\mathfrak{h};\mathbb R)$. Thus, since they are finite dimensional, $H^*(\mathfrak{h};\mathbb R)$ and $H^*(\mathfrak{g};\mathbb R)$ are isomorphic algebras, as desired.
\smallskip

\subsection*{Acknowledgments} We thank Roman Sauer for having shared with us his notes with Arthur Bartels on $C_\infty$ algebras and morphisms. We also thank Kyle Hansen, Bruce Kleiner, Enrico Le Donne, Fedya Manin, Assaf Naor, and Stefan Wenger for useful discussions and encouragements during various stages of this project. {We thank the anonymous reviewer for their careful reading of the paper and useful suggestions.}  G.A. acknowledges the financial support of the Courant Institute and the AMS-Simons Travel Grant. R.Y. was supported by the National Science Foundation under Grant No.\ 2005609.
\smallskip

Data sharing not applicable to this article as no datasets were generated or analyzed during the
current study.
\smallskip

\textbf{Conflict of interest}. The authors state that there is no conflict of interest.

\section{Preliminaries}\label{sec:bounded-derivs}
Let $G$ and $H$ be simply connected nilpotent Lie groups. We denote the identities of $G$ and $H$ by $\id_G$ and $\id_H$.  Let $\mathfrak{g}$ and $\mathfrak{h}$ be their Lie algebras; we identify $\mathfrak{g}$ with the tangent space $T_{\id_G}G$ and with the space $\mathcal{L}_G$ of left-invariant vector fields on $G$. 

We fix Haar measures on $G$ and $H$. We denote the Haar measure of a subset $U\subset G$ by $|U|$ and the integral with respect to Haar measure by $\int_G \beta(g) \ud g$. We denote the average with respect to Haar measure by
$$
\fint_{U}\beta \ud g := \frac{1}{|U|} \int_{U} \beta \ud g.
$$

We will need to define some spaces of smooth maps with bounded derivatives. 
Let $C^{\infty, \infty}(G)\subset C^\infty(G)$ be the space of smooth functions $\alpha\from G\to \C$ which have bounded derivatives in the sense that $\mathcal{D}\alpha\in L^\infty(G)$ for any left-invariant differential operator $\mathcal{D}$. That is, if $V_1,\dots, V_m \in \mathcal{L}_G$ is a basis of the Lie algebra of $G$, then
$$V_{i_1}\dots V_{i_d} \alpha \in L^\infty(G)$$
for any multi-index $I=(i_1,\dots,i_d)$.

Let $U_1,\dots, U_{m'}\in \mathfrak{h}$ be a basis of the Lie algebra of $H$. Suppose that $\alpha \in C^\infty(G,H)$. Let $m_{ij}\from G\to \R$ be the matrix coefficients of $D_g\alpha\from T_gG\to T_{\alpha(g)}H$ , i.e,
$$D_g\alpha(V_i(g)) = \sum_j m_{ij}(g) U_j(\alpha(g)).$$
We say that $\alpha \in C^{\infty,\infty}(G,H)$ if $m_{ij}\in C^{\infty,\infty}(G)$ for all $i$ and $j$; if so, we say that $\alpha$ has \emph{bounded derivatives}. 
Let
\begin{equation}\label{eq:def-bounded}
  Y_0(G,H) := \{\alpha \in C^{\infty,\infty}(G,H) : \alpha(\id_G)=\id_H\}.
\end{equation}
Note that $Y_0(G,H)\subset \Lip(G,H)$. We equip $C^{\infty,\infty}(G)$ and $Y_0(G,H)$ with the Whitney topology, so that a sequence $\alpha_1,\alpha_2,\dots$ converges to $\alpha$ if and only if for every compact $K\subset G$, and every $n\in\mathbb N$, we have $\alpha_i\to \alpha$ in the $C^n$--topology on $K$ for all $n$. 
This makes $C^{\infty,\infty}(G)$ and $Y_0(G,H)$ locally compact Hausdorff spaces. In fact,  $C^{\infty,\infty}(G)$ and $Y_0(G,H)$ are metrizable. Let $\mathcal{D}_1, \mathcal{D}_2,\dots$ be a countable basis of the set of left-invariant differential operators, and $K_1\subset K_2\subset \dots$ be an exhaustion of compact sets in $G$. For $\alpha,\beta\in C^{\infty,\infty}(G)$, and $i,j\geq 1$, define $p_{i,j}(\alpha,\beta):=\|\mathcal{D}_i[\beta-\alpha]\|_{L^\infty(K_j)}$. Then, re-indexing $\{p_{i,j}\}_{i,j\geq 1}=\{p_k\}_{k\geq 1}$, we have that
$$
d_{\infty,\infty}(\alpha,\beta) = \sum_{k=1}^\infty \min\left\{p_k(\alpha,\beta), 2^{-k}\right\},
$$
is a metric on $C^{\infty,\infty}(G)$. A similar formula defines a metric on $Y_0(G,H)$.

\section{Ergodic maps}\label{sec:ergodic}

\subsection{Definitions}
Let $G$ and $H$ be simply connected nilpotent Lie groups and let $Y_0=Y_0(G,H)$ be as in Section~\ref{sec:bounded-derivs}.
For $g\in G$, and $h\in H$, let $l_g\from G\to G$ and $l_h\from H\to H$ be the left-translation maps.

For any $g\in G$ and $\alpha\in Y_0$, let 
\begin{equation}\label{eq:tau-lambda}
\alpha\cdot g := l_{\alpha(g)}^{-1}\circ\alpha\circ l_g.
\end{equation}
One can check that $\alpha\cdot g \in Y_0$ and that for $g_1,g_2\in G$,
\begin{multline*}
  (\alpha \cdot g_1) \cdot g_2 = l_{(\alpha\cdot g_1)(g_2)}^{-1}\circ l_{\alpha(g_1)}^{-1} \circ \alpha \circ l_{g_1g_2} \\ = l_{\alpha(g_1)^{-1}\alpha(g_1g_2)}^{-1}\circ l_{\alpha(g_1)}^{-1} \circ \alpha \circ l_{g_1g_2} 
  = \alpha\cdot (g_1g_2).
\end{multline*}
That is, this is a right action of $G$ on $Y_0$. Note that the derivatives of $\alpha\cdot g$ satisfy bounds that depend only on $\alpha$, so for any $\alpha\in Y_0$, the closure $\clos(\alpha\cdot G)$ is compact by the Arzelà--Ascoli Theorem.

\begin{defn}\label{def:ErgodicMap}
  Let $y \in Y_0$. For every $n>0$, there is a measure $\mu_n$ on $Y_0$ such that for any continuous function $A \from Y_0\to \R$,
  \begin{equation}\label{eq:def-mu-n}
    \int A \ud \mu_n := \fint_{B_n} A(y\cdot g)\ud g.
  \end{equation}
 If the $\mu_n$'s converge weakly to a measure $\mu$ on $Y_0$, we call $\mu$ the \emph{limit measure} of $y$. Since the orbit $y\cdot G$ has compact closure, if $\mu$ exists, then $\mu$ is a probability measure supported on a compact set, and the amenability of $G$ implies that $\mu$ is $G$--invariant. We say that $y$ is an \emph{ergodic map} if $\mu$ exists and is ergodic with respect to the action of $G$.
\end{defn}

\subsection{Constructing ergodic maps}
We construct ergodic maps using the Krein--Milman theorem and the following theorem of E.\ Lindenstrauss \cite[Theorem 1.2]{Lindenstrauss}. 
\begin{thm}[Lindenstrauss]\label{thm:lindenstrauss}
  Let $G$ be an amenable group equipped with a left-invariant Haar measure and let $(X,\eta)$ be a probability space equipped with an ergodic left action of $G$. Let $F_n\subset G$ be a tempered Følner sequence. Then for any $f\in L_1(\eta)$, the following holds for \(\eta\)--a.e. $x\in X$,
  $$\lim_{n\to \infty} \fint_{F_n} f(g\cdot x) \ud g = \int_X f \ud \eta.$$
\end{thm}
A \emph{tempered Følner sequence} $F_n$ is a Følner sequence for $G$ such that there is some $C > 0$ such that for all $n$,
\begin{equation}\label{eq:tempered}
  \left| \bigcup_{k\le n} F_k^{-1} F_{n+1}\right| \le C |F_{n+1}|.
\end{equation}
In all our uses of Theorem~\ref{thm:lindenstrauss}, {$G$ will be a nilpotent group, and} we will take $F_n = B_n(\id_G)$, which satisfies \eqref{eq:tempered}. Throughout the paper we will abbreviate $B_n:=B_n(\id_G)$.

To prove Proposition~\ref{prop:ExistenceOfErgodic}, we first construct an invariant ergodic probability measure on $Y_0$. {We stress once more that from now on we adopt the notation set in Section \ref{sec:bounded-derivs}, so that $G,H$ are simply connected nilpotent Lie groups.}
\begin{lemma}\label{lem:ergodic-closure}
  Let $f\in Y_0$. Then there is a $G$--invariant ergodic probability measure $\mu$ on $Y_0$ with support contained in the closure $K = \clos(f\cdot G)$.
\end{lemma}
\begin{proof}
  We first construct an invariant probability measure $\eta$ supported on $K$.  Let $C_c(Y_0)$ be the set of continuous, compactly-supported, complex-valued functions on $Y_0$.  Let $\ulim$ be an ultralimit, and for any $\omega \in C_c(Y_0)$, let
  $$
  \alpha(\omega) := \ulim_{n\to \infty} \fint_{B_n}\omega(f \cdot g)\ud g.$$
  By the amenability of $G$, $\alpha$ is $G$--invariant.
  By the Riesz representation theorem, since $Y_0$ is locally compact and Hausdorff, there is a $G$--invariant Radon measure $\eta$ on $Y_0$ such that $\alpha(\omega) = \int_{Y_0} \omega \ud \eta$ for all $\omega\in C_c(Y_0)$. 

  Since $K$ is compact and $Y_0$ is metrizable, there is a sequence of functions $k_i\in C_c(Y_0)$ such that $k_i=1$ on $K$ and $\supp(k_i) \to K$; then $\alpha(k_i)=1$ for all $i$, so $\eta(K)=1$. Furthermore, if $\omega\in C_c(Y_0)$ has support disjoint from $K$, then $\alpha(\omega)=0$, so $\eta$ is a probability measure supported on $K$.

  We can view the set $P$ of $G$--invariant probability measures supported on $K$ as a convex subset of the dual space $C(K)^*$, equipped with the weak-$*$ topology. Since $K$ is compact, this set is compact; since it contains $\eta$, it is nonempty. By the Krein--Milman Theorem, the set $E(P)$ of extreme points of $P$ is nonempty; let $\mu \in E(P)$. By extremality, $\mu$ is an ergodic probability measure supported on $K$, as desired.
\end{proof}

We use this lemma to prove Proposition~\ref{prop:ExistenceOfErgodic}.
\begin{proof}[Proof of Proposition~\ref{prop:ExistenceOfErgodic}]
  Let $f\in Y_0$, let $K = \clos(f\cdot G)$, and let $\mu$ be a $G$--invariant ergodic probability measure on $K$ as in Lemma~\ref{lem:ergodic-closure}.

  Since $Y_0$ is metrizable, there is a countable dense set $S\subset C(K)$. For every $s\in S$, Theorem~\ref{thm:lindenstrauss} implies that there is a 
$K_s\subset K$ such that $\mu(K\setminus K_s)=0$ and such that for any $y \in K_s$,
  \begin{equation}
    \lim_{n\to \infty} \fint_{B_n} s(y \cdot g^{-1}) \ud g = \lim_{n\to \infty} \fint_{B_n} s(y \cdot g) \ud g = \int_{Y_0} s\ud \mu.
  \end{equation}
  Since $S$ is countable, $\bigcap_s K_s$ is nonempty; let $\phi \in \bigcap_s K_s$. Then for any $A\in C(K)$, 
  \begin{equation}\label{eqn:Gamma}
    \lim_{n\to \infty} \fint_{B_n} A(\phi \cdot g) \ud g = \int_{Y_0} A\ud \mu,
  \end{equation}
  i.e., $\phi$ is ergodic with limit measure $\mu$.

  Furthermore, since $\phi\in K=\clos(f\cdot G)$, there is a sequence $g_1,g_2,\dots$ such that $f\cdot g_n$ converges to $\phi$ on compact subsets, as desired. 
  
  {Let us now prove the last part of the Proposition. By left-invariance of the distance, and the very definition of $f\cdot g_n$ in \eqref{eq:tau-lambda}, we get that, for every $n\in\mathbb N$, $f\cdot g_n$ is a quasi-isometric embedding with the same constants of $f$ whenever $f$ is a quasi-isometric embedding, see \eqref{eqn:LCQiEmbedding}. Furthermore, since $f\cdot g_n$ converges to $\phi$ on compact subsets, $\phi$ verifies \eqref{eqn:LCQiEmbedding} with the same constants of $f$ whenever $f$ is a quasi-isometric embedding. As a result, if $f$ is a quasi-isometric embedding, so is $\phi$, as desired. 
  
  Let us now assume that in addition $f$ is a quasi-isometry, see \eqref{eqn:LCQiEmbedding}. Fix $h\in H$, and take $g'_n$ such that $d(f(g'_n),f(g_n)h)\leq C$. By definition of $f\cdot g_n$, calling $\hat g_n:=(g_n)^{-1}g_n'$, we get $d((f\cdot g_n)(\hat g_n),h)\leq C$. Since $d(f(g'_n),f(g_n)h)\leq C$ and $f$ is a quasi-isometric embedding we have that $d(\hat g_n,\mathrm{id}_G)$ is bounded from above independently on $n$. Thus, up to subsequences, $\hat g_n\to \tilde g$. Hence, passing to the limit $d((f\cdot g_n)(\hat g_n),h)\leq C$, we get $d(\phi(\tilde g),h)\leq C$, thus proving that $\phi$ is a quasi-isometry, as desired.}
\end{proof}

\section{The universal pullback map}\label{sec:Pullbacks}
{Let $\phi\in Y_0$}. For any left-invariant form $\omega\in \bigwedge^*\mathfrak{h}$, the pullback $\phi^*\omega$ is a differential form on $G$ which is bounded, but generally not left-invariant.

One of the main ideas of the proof of Theorem~\ref{thm:mainThm} is that we can write such pullbacks in terms of a space of functions $A := C^{\infty,\infty}(Y_0)\subset C(Y_0)$ (where $C(Y_0)$ is the set of continuous functions from $Y_0$ to $\C$), and the corresponding differential graded algebra $C^*(\mathfrak{g};A)$.

Each element $\alpha\in A$ will correspond to a $G$--equivariant operator that sends an element $\phi\in Y_0$ to a smooth function $\alpha_\phi \in C^{\infty,\infty}(G)$, and each $\beta\in C^*(\mathfrak{g};A)$ will correspond to a $G$--equivariant operator that sends an element $\phi\in Y_0$ to a differential form $\beta_\phi$. We will use this correspondence to construct the following universal pullback map.
\begin{lemma}\label{lem:Tsharp}
  There is a morphism of differential graded algebras $T^\sharp \from C^*(\mathfrak{h}) \to C^*(\mathfrak{g};A)$ such that for every $\omega\in C^k(\mathfrak{h})$ and $\phi\in Y_0$,
  \begin{equation}\label{eqn:ImportantEquality}
    (T^\sharp\omega)_\phi = \phi^*\omega.
  \end{equation} 
\end{lemma}

Then, in the next Section \ref{sec:VanishingLemma}, we will give conditions for $T^\sharp\omega$ to be a limit of coboundaries, i.e., a reduced coboundary.

\subsection{Equivariant operators and $C^{\infty,\infty}(Y_0)$}
Let $C(Y_0)$ and $C(G)$ be the spaces of complex-valued continuous functions on $Y_0$ and $G$ respectively, equipped with the topology of uniform convergence on compact sets. 
If $\mathcal{D}\from Y_0 \to C(G)$, we say that $\mathcal{D}$ is \emph{$G$--equivariant} if for all $g, h\in G$ and $\phi\in Y_0$,
\begin{equation}\label{eq:equivar-op}
  \mathcal{D}[\phi \cdot g](h) = \mathcal{D}[\phi](gh),
\end{equation}
i.e., $\mathcal{D}[\phi \cdot g] = \mathcal{D}[\phi] \circ l_g$. Then $\mathcal{D}[\phi]$ has the same symmetries as $\phi$; if $g\in G$ and $\phi\cdot g = \phi$, then $\mathcal{D}[\phi] \circ l_g = \mathcal{D}[\phi]$.

{ Let $\mathsf{Op}$ be the set of continuous $G$--equivariant operators $\mathcal{D}\from Y_0 \to C(G)$.} We define $\mathsf{Op}^{\infty,\infty}\subset \mathsf{Op}$ to be the set of \emph{smooth} $G$--equivariant operators, i.e., operators $\mathcal{D}\from Y_0\to C^{\infty,\infty}(G)$ such that $\mathcal{D}$ is continuous with respect to the topology on $C^{\infty,\infty}(G)$ and that satisfy \eqref{eq:equivar-op}.

The most important elements of $\mathsf{Op}^{\infty,\infty}$ in this paper correspond to pullbacks of left-invariant forms on $H$.
For example, if $V\in \wedge^d \mathfrak{g}$ is a left-invariant field of $d$--vectors on $G$ and $\omega \in C^d(\mathfrak{h})$ is a left-invariant $d$--form, then
\begin{equation}\label{eq:pullback-example}
  \mathcal{D}[\phi] := \phi^*\omega(V) = \omega(\phi_*(V))
\end{equation}
is smooth and satisfies \eqref{eq:equivar-op}, so $\mathcal{D}\in \mathsf{Op}^{\infty,\infty}$.

We will use $\mathsf{Op}^{\infty,\infty}$ to define $C^{\infty,\infty}(Y_0)$, an algebra of functions on $Y_0$. First, we can represent elements of $\mathsf{Op}$ as continuous functions on $Y_0$.
\begin{lemma}\label{lem:c-bijection}
Let $P\from \mathsf{Op} \to C(Y_0)$, 
  $$P(\mathcal{D})(\phi) := \mathcal{D}[\phi](\id_G).$$
  Then $P$ is a bijection with inverse $P^{-1}(\alpha) = \mathcal{D}_\alpha$, where for any $\alpha\in C(Y_0)$, we let $\mathcal{D}_\alpha\from Y_0\to C(G)$ be the map
  \begin{equation}\label{eq:def-dalpha}
    \mathcal{D}_\alpha[\phi](h) := \alpha(\phi \cdot h), \text{\qquad} \phi\in Y_0, h\in G.
  \end{equation}
\end{lemma}
\begin{proof}
  For any $g,h\in G$, $\alpha\in C(Y_0)$, and $\phi\in Y_0$,
  $$\mathcal{D}_\alpha[\phi\cdot g](h) = \alpha(\phi \cdot gh) = \mathcal{D}_\alpha[\phi](gh),$$
  so $\mathcal{D}_\alpha\in \mathsf{Op}$. Furthermore, for $\alpha\in C(Y_0)$ and $\phi\in Y_0$,
  $$
  P(\mathcal{D}_\alpha)(\phi) =  \mathcal{D}_\alpha[\phi](\id_G) = \alpha(\phi),
  $$
  and for $\mathcal{D}\in \mathsf{Op}$,
  $$\mathcal{D}_{P(\mathcal{D})}[\phi](h) = P(\mathcal{D})(\phi\cdot h) = \mathcal{D}[\phi\cdot h](\id_G) = \mathcal{D}[\phi](h),$$
  so $\mathcal{D}_{P(\mathcal{D})} = \mathcal{D}$.
\end{proof}

Let $C^{\infty,\infty}(Y_0) := P(\mathsf{Op}^{\infty,\infty})$. Since elements of $C^{\infty,\infty}(Y_0)$ correspond to operators from $Y_0$ to $C^{\infty,\infty}(G)$, we introduce some notation.
For $\alpha\in C(Y_0)$ and $\phi\in Y_0$, we let $\alpha_\phi := \mathcal{D}_\alpha[\phi] \in C(G),$
i.e., 
\begin{equation}\label{eqn:DefAlphay}
\alpha_\phi(g) := \alpha(\phi\cdot g) \text{\qquad for all }g\in G.
\end{equation}
Then {by \eqref{eqn:DefAlphay}, and the fact that $\cdot$ is a right action,}
\begin{equation}\label{eqn:alphay-sym}
  \alpha_{\phi \cdot g}(h) = \alpha_{\phi}(gh) \text{\qquad for all }g,h\in G.
\end{equation}
By the following proposition, $C^{\infty,\infty}(Y_0)$ is an algebra.
\begin{prop}\label{prop:PropertiesOfCinftyinfty}
    The following properties hold:
    \begin{enumerate}
    \item $C^{\infty,\infty}(Y_0)$ is an algebra over $\C$, i.e., it is closed under addition, multiplication, and scalar multiplication.
    \item For every $\alpha\in C^{\infty,\infty}(Y_0)$ and every $X\in \mathfrak{g}$, there is an element $X\alpha \in C^{\infty,\infty}(Y_0)$ such that for any $\phi\in Y_0$,
      \begin{equation}\label{eq:Xalpha-def}
        (X\alpha)_\phi = X[\alpha_\phi].
      \end{equation}
      Furthermore, for any $X,Y\in \mathfrak{g}$ and $\alpha\in C^{\infty,\infty}(Y_0)$,
      \begin{equation}\label{eq:lie-rep}
        [X,Y]\alpha = X[Y \alpha] - Y[X \alpha],
      \end{equation}
      so $C^{\infty,\infty}(Y_0)$ is a $\mathfrak{g}$--module. 
    \item For every $\alpha \in C(Y_0)$, there is a sequence $(\alpha_n)_n \in C^{\infty,\infty}(Y_0)$ such that for any compact $K\subset Y_0$,
      $$\lim_{n\to \infty} \|\alpha_n-\alpha\|_{L^\infty(K)} = 0.$$
    \end{enumerate}
\end{prop}
\begin{proof}  
  Because of the bijection between $C^{\infty,\infty}(Y_0)$ and $\mathsf{Op}^{\infty,\infty}$, we can prove the proposition using either set.

  For item (1), note that for $\alpha,\beta\in C(Y_0)$, every $y\in Y_0$, and every $z\in\mathbb C$ we have $(\alpha+\beta)_y = \alpha_y+\beta_y$, $(\alpha \beta)_y = \alpha_y\beta_y$, and $(z\alpha)_y=z\alpha_y$, so $C^{\infty,\infty}(Y_0)$ is closed under addition, multiplication, and scalar multiplication. 

  To prove item (2), for any left-invariant vector field $X$, $\mathcal{D}\in \mathsf{Op}^{\infty,\infty}$, and $\phi\in Y_0$, let 
  $$
  X\mathcal{D}[\phi] := X[\mathcal{D}[\phi]].
  $$
  It is straightforward to check that $X\mathcal{D}\in \mathsf{Op}^{\infty,\infty}$.

  Furthermore, for $X,Y\in \mathfrak{g}$, $\mathcal{D}\in \mathsf{Op}^{\infty,\infty}$, and $\phi\in Y_0$,
  $$
  [X,Y]\mathcal{D}[\phi]= [X,Y][\mathcal{D}[\phi]] = XY[\mathcal{D}[\phi]] - YX[\mathcal{D}[\phi]] = X[Y\mathcal{D}][\phi] - Y[X\mathcal{D}][\phi],
  $$
  proving \eqref{eq:lie-rep}.

  Finally, we prove item (3). For $h\in C(G)$ and $k\in C_c^\infty(G)$, let $h\ast k \in C^\infty(G)$,
  \begin{equation}\label{eqn:g-Convolution}
    (h\ast k)(x) := \int_{G} h(y) k(y^{-1} x) \ud y.
  \end{equation}  
  
  Let $\mathcal{D}\in \mathsf{Op}$ and $k\in C_c^\infty(G)$. 
  For $\phi \in Y_0$, let
  $$(\mathcal{D}\ast k)[\phi] = \mathcal{D}[\phi] \ast k \in C^{\infty,\infty}(G).$$
  It is readily verified that $\mathcal{D}\ast k$ is $G$--equivariant, so $\mathcal{D}\ast k\in \mathsf{Op}^{\infty,\infty}$. Thus, by Lemma~\ref{lem:c-bijection}, for any $\alpha\in C(Y_0)$, we can define $\alpha\ast k\in C^{\infty,\infty}(Y_0)$ { as $\alpha\ast k = P(\mathcal{D}_\alpha\ast k)$,  so that $\mathcal{D}_{\alpha \ast k} = \mathcal{D}_\alpha\ast k$}.

  Let $(k_n)_n \in C_c^\infty(G)$ be a sequence of nonnegative functions such that \[\int_G k_n(x)\ud x = 1,\] and $\supp(k_n)\subset B_{\frac{1}{n}}(\id_G)$ for all $n$, so that for any $h\in C(G)$, $h\ast k_n\to h$ uniformly on compact sets as $n\to \infty$. 
    
  Let $\alpha_n := \alpha\ast k_n$, so that for any $\phi\in Y_0$, we have
  $$\alpha_n(\phi) = (\alpha_\phi \ast k_n)(\id_G).$$
  and thus $\lim_{n\to \infty} \alpha_n(\phi) = \alpha_\phi(\id_G) = \alpha(\phi)$, i.e., $\alpha_n$ converges to $\alpha$ pointwise.

  We claim that $\alpha_n$ converges to $\alpha$ uniformly on compact sets. Let $K\subset Y_0$ be a compact set. The function $(\phi,g)\mapsto \alpha_\phi(g)$ is continuous, so it is uniformly continuous in $K \times \overline{B}_1(\id_G)$. That is, the functions $\alpha_\phi, \phi\in K$, are uniformly continuous with a uniform modulus of continuity. Thus, for any $\varepsilon>0$, there is an $N>0$ such that for all $\phi\in K$ and $n>N$, 
  $$|\alpha_n(\phi)-\alpha(\phi)| = |(\alpha_\phi \ast k_n)(\id_G) - \alpha_\phi(\id_G)| < \varepsilon,$$
  i.e., $\|\alpha_n-\alpha\|_{L^\infty(K)}\to 0$ as $n\to \infty$. { In the previous inequality we are using the uniform continuity of $\alpha_\phi$ and our choice of the support of $k_n$}. This proves item (3).
\end{proof}

\subsection{Differential forms with twisted coefficients and the universal pullback map}\label{sec:diff-forms}

Next, we define a complex of differential forms with coefficients in $C^{\infty,\infty}(Y_0)$. Let $A=C^{\infty,\infty}(Y_0)$, and for $k\ge 0$ let
\[
C^k(\mathfrak{g};A) := \Hom(\wedge^k \mathfrak{g}, A).
\]
Any cochain $\beta\in C^k(\mathfrak{g};A) = \Hom(\wedge^k \mathfrak{g}, A)$ corresponds to a function $\wedge^k \mathfrak{g}\times Y_0 \to \C$. If we view $\beta$ instead as a function $Y_0 \to \Hom(\wedge^k \mathfrak{g},\C)$, we obtain a map $V_\beta \from Y_0 \to \wedge^k \mathfrak{g}^*$ such that $V_\beta(\phi)(\lambda) = \beta(\lambda)(\phi)$ for all $\phi\in Y_0$ and $\lambda\in \wedge^k\mathfrak{g}$. As before, we can define an operator $\mathcal{D}_\beta \from Y_0 \to \Omega^k(G)$ by letting $\mathcal{D}_\beta[\phi]$ be the $k$--form whose value at any $g\in G$ is
$$\mathcal{D}_\beta[\phi]_g := V_\beta(\phi\cdot g).$$
Then for any left-invariant vector fields $v_1,\dots, v_k\in \mathfrak{g}$, 
\begin{equation}\label{eq:def-form-eval}
  \mathcal{D}_\beta[\phi](v_1,\dots,v_k) = \mathcal{D}_{\beta(v_1,\dots,v_k)}[\phi].
\end{equation}
For conciseness, let $\beta_\phi := \mathcal{D}_\beta[\phi]$.

Since $A$ is a $\mathfrak{g}$--module, we can define a differential and wedge product on $C^*(\mathfrak{g};A)$ by a standard construction: for every $f\in C^k(\mathfrak{g};A)$, and every $X_1,\dots,X_{k+1}\in\mathfrak{g}$, let
\begin{multline}\label{eqn:ExteriorDifferential}
  \mathrm{d}f (X_1,\dots,X_{k+1}):=\sum_{i=1}^{k+1} (-1)^{i+1} X_i\left[f(X_1,\dots,\hat X_i,\dots,X_{k+1})\right]\\
  +\sum_{i<j}(-1)^{i+j}f([X_i,X_j],\dots,\hat X_i,\dots,\hat X_j,\dots,X_{k+1}),
\end{multline}
and for every $\alpha\in C^m(\mathfrak{g};A),\beta\in C^n(\mathfrak{g};A)$, let
\begin{equation}\label{eqn:DefnCupProduct}
  \alpha\wedge\beta:=\frac{(m+n)!}{m!n!}\mathcal{A}(\alpha\otimes\beta)\in C^{m+n}(\mathfrak{g};A),
\end{equation}
where $\alpha\otimes\beta$ is the tensor product and $\mathcal{A}$ is projection on alternating forms, i.e., for every $(0,k)$-tensor $t$, and every $k$-tuple  $X_1,\ldots,X_k\in\mathfrak{g}$
\[
  \mathcal{A}t(X_1,\ldots,X_k):=\frac{1}{k!}\sum_{\sigma \in S_k}\operatorname{sgn}(\sigma)t(X_{\sigma(1)},\ldots,X_{\sigma(k)}),
\]
where $\operatorname{sgn}(\sigma)$ is the signature of $\sigma$.

It is straightforward to check that for any $\alpha, \beta \in C^*(\mathfrak{g};A)$, and $\phi\in Y_0$,
\begin{equation}\label{eq:preserve-diff}
  (\mathrm{d} \alpha)_\phi = \mathrm{d}[\alpha_\phi],
\end{equation}
and
\begin{equation}\label{eq:preserve-wedge}
  (\alpha \wedge \beta)_\phi = \alpha_\phi\wedge \beta_\phi.
\end{equation}
Consequently, one can verify that $C^*(\mathfrak{g};A)$ is a differential graded algebra either directly or by noting that
$$
\mathrm{d}[\mathrm{d}\alpha]_\phi=\mathrm{d}[[\mathrm{d} \alpha]_\phi] = \mathrm{d}[\mathrm{d}[\alpha_\phi]] = 0,
$$
and
\begin{equation}\label{eqn:LiebnizCup}
\mathrm{d}[\alpha\wedge \beta]_\phi = (\mathrm{d}\alpha \wedge \beta + (-1)^{|\alpha|} \alpha \wedge \mathrm{d}\beta)_\phi,
\end{equation}
for all $\phi$.

We use this structure to construct the map $T^\sharp \from C^*(\mathfrak{h}) \to C^*(\mathfrak{g};A)$ and prove Lemma~\ref{lem:Tsharp}.
\begin{proof}[Proof of Lemma~\ref{lem:Tsharp}]
  For every $\omega\in C^k(\mathfrak{h})$ and $\lambda\in \wedge^k \mathfrak{g}$, let
  $$
  \mathcal{D}_{\omega,\lambda}[\phi] := \phi^*\omega(\lambda)\in C^{\infty,\infty}(G).
  $$
  This operator is continuous and satisfies \eqref{eq:equivar-op}, so $\mathcal{D}_{\omega,\lambda}\in \mathsf{Op}^{\infty,\infty}$. {By Lemma~\ref{lem:c-bijection}, there is a $t(\omega,\lambda)\in A$ such that 
  \[
  t(\omega,\lambda)_\phi  =\mathcal{D}_{\omega,\lambda}[\phi]= \phi^*\omega(\lambda),
  \]
  for all $\phi\in Y_0$.
 
  For every $\omega\in C^k(\mathfrak{h})$, we define $T^\sharp\omega\in C^k(\mathfrak{g};A)$ so that $T^\sharp\omega(\lambda):=t(\omega,\lambda)$ for every $\lambda\in\Lambda^k\mathfrak{g}$. Since $\phi^*\omega(\lambda)$ is linear in $\lambda$, $T^\sharp\omega$ is an $A$--valued form, i.e., $T^\sharp\omega\in C^k(\mathfrak{g};A)$. Now, by using the previous equality, and \eqref{eq:def-form-eval} with $\beta=T^\sharp\omega$, we have
  \[
   (T^\sharp\omega)_\phi(\lambda) \stackrel{\eqref{eq:def-form-eval}}{=} 
  (T^\sharp\omega(\lambda))_\phi = \phi^*\omega(\lambda),
  \]
  i.e., $T^\sharp$ satisfies \eqref{eqn:ImportantEquality}.}

  In order to show that $T^\sharp$ is a morphism of differential graded algebra, it remains to check that $T^\sharp$ commutes with the wedge product and differential. For every $\omega,\eta\in C^k(\mathfrak{h})$ and $\phi\in Y_0$, we have
  {
  \[
(T^\sharp(\omega\wedge\eta))_\phi \stackrel{\eqref{eqn:ImportantEquality}}{=} \phi^*(\omega\wedge\eta) = \phi^*\omega\wedge \phi^*\eta = (T^\sharp\omega)_\phi\wedge (T^\sharp\eta)_\phi \stackrel{\eqref{eq:preserve-wedge}}{=} (T^\sharp\omega \wedge T^\sharp\eta)_\phi.
  \]
  }
  Hence $T^\sharp(\omega\wedge\eta) = T^\sharp\omega\wedge T^\sharp\eta$.

  Similarly, for every $\omega\in C^k(\mathfrak{h})$, and every $\phi\in Y_0$,
  {
  \[
    (T^\sharp (\mathrm{d}\omega))_\phi \stackrel{\eqref{eqn:ImportantEquality}}{=} \phi^*(\mathrm{d}\omega)=\mathrm{d}[\phi^*\omega]=\mathrm{d}[(T^\sharp\omega)_\phi]\stackrel{\eqref{eq:preserve-diff}}{=}(\mathrm{d}[T^\sharp\omega])_\phi,
  \]
  }
  thus the proof is complete.
\end{proof}

\section{Amenable averages and the vanishing lemma}\label{sec:VanishingLemma}

In the previous sections, we showed that any ergodic map $\psi\in Y_0$ has a limit measure $\mu$ on $Y_0$ and we constructed a universal pullback map $T^\sharp\from H^*(\mathfrak{h})\to H^*(\mathfrak{g};A)$, where $A=C^{\infty,\infty}(Y_0)$ is a space of continuous functions on $Y_0$ corresponding to smooth $G$-equivariant operators $Y_0\to C^{\infty,\infty}(G)$. {We actually defined $T^\sharp\from C^*(\mathfrak{h})\to C^*(\mathfrak{g};A)$, but by Lemma \ref{lem:Tsharp}, this descends to a map $T^\sharp\from H^*(\mathfrak{h})\to H^*(\mathfrak{g};A)$.} In this section, we show that elements of $C^*(\mathfrak{g};A)$ whose \emph{amenable average} is zero are reduced coboundaries (i.e., limits of coboundaries) in a larger space $C^*(\mathfrak{g};L)$.

We first define the amenable average of a cochain. For any bounded form $\omega\in \Omega^k(G)$ and $p\in G$, denote the value of $\omega$ at $p$ by $\omega_p\in \wedge^k \mathfrak{g}^*$. We define the amenable average
$\overline{\omega} \in \wedge^k \mathfrak{g}^*$ by
\begin{equation}\label{eq:def-amen-avg}
  \overline{\omega} = \lim_{n\to \infty} \fint_{B_n} \omega_g \ud g,
\end{equation}
if the limit exists.

In the rest of this paper, we fix an ergodic map $\psi\in Y_0$ and let $\mu$ be its limit measure. When $\omega = \beta_\psi$, we can describe $\overline{\omega}$ in terms of $\mu$.
\begin{lemma}\label{lem:link-amenable}
  Let $\beta \in C^k(\mathfrak{g};A)$. Let $V_\beta \from Y_0 \to \wedge^k \mathfrak{g}^*$ as in Section~\ref{sec:diff-forms}. Then
  $$\overline{\beta_\psi} = \int V_\beta \ud \mu.$$
\end{lemma}
\begin{proof}
  As in Definition~\ref{def:ErgodicMap}, for $n\ge 0$, let $\mu_n$ be the measure on $Y_0$ such that 
  $$\int \alpha \ud \mu_n = \fint_{B_n} \alpha(\psi \cdot g) \ud g$$
  for any continuous $\alpha\in C(Y_0)$, and $\mu_n$ converges weakly to $\mu$. In particular, since $V_\beta\from Y_0\to \wedge^k \mathfrak{g}^*$ is continuous,
  \begin{equation*}
    \lim_{n\to \infty} \fint_{B_n} (\beta_\psi)_g \ud g = \lim_{n\to \infty} \fint_{B_n} V_\beta(\psi\cdot g) \ud g = \lim_{n\to \infty} \int V_\beta \ud \mu_n = \int V_\beta \ud \mu.\qedhere
  \end{equation*}
\end{proof}

In particular, $\overline{\beta_\psi}$ is determined by the values of $V_\beta$ on $\supp(\mu)$, which encourages us to study the space $L^2(\mu)$. 
We equip $L^2(\mu)$ with the left action $g\cdot f(\phi) = f(\phi\cdot g)$ for all $\phi\in Y_0$ and $g\in G$. For any $v\in L^2(\mu)$, let $\eta_v\from G\to L^2(\mu)$, $\eta_v(g) = g\cdot v$.
{We say that $\eta_v$ is smooth if all of its Fréchet derivatives of all orders exist and are continuous. 
\begin{lemma}\label{lem:cinfty-in-L}
  Let $\alpha\in C^{\infty,\infty}(Y_0)$. Then $\alpha\in L^2(\mu)$ and $\eta_\alpha$ is smooth. In fact, for all $X\in\mathfrak{g}$, if $X\alpha$ is as in Proposition~\ref{prop:PropertiesOfCinftyinfty}, then $X[\eta_\alpha] = \eta_{X\alpha}$.
\end{lemma}
\begin{proof}
  Since $\mu$ is supported on a compact set, $\alpha\in L^2(\mu)$. Furthermore, 
  $$\eta_\alpha(g)(\phi)=g\cdot \alpha(\phi) = \alpha(\phi\cdot g) = \alpha_\phi(g).$$
  By Proposition~\ref{prop:PropertiesOfCinftyinfty}, $X[\alpha_\phi] = (X\alpha)_\phi$, so 
  $$
  X[\eta_\alpha](g)(\phi)=X[\alpha_\phi](g) = (X\alpha)_\phi(g) = \eta_{X\alpha}(g)(\phi),
  $$
  as desired.
\end{proof}}

{We let
\begin{equation}\label{eqn:L2infty}
L=(L^2(\mu))^\infty := \{v\in L^2(\mu) : \eta_v \text{ is smooth}\},
\end{equation}}
and put a $\mathfrak{g}$--module structure on $L$ by letting
\begin{equation}\label{eq:module-structures}
  X v := X\eta_v (\id_G) = \frac{\mathrm{d}}{\mathrm{d} t}[\exp(tX) \cdot v]_{t=0} \in L,
\end{equation}
for all $X\in \mathfrak{g}$, and $v\in L$. {By Lemma~\ref{lem:cinfty-in-L}, this extends to $L$ the $\mathfrak{g}$--module structure on $C^{\infty,\infty}(Y_0)$.}
Using this structure, we define the cochain complex $C^*(\mathfrak{g};L)$ in the usual way (see Section~\ref{sec:diff-forms}).

We equip $L$ with the topology such that $\lim_{n\to \infty} v_n = v$ if and only if 
$$
\lim_{n\to \infty} \eta_{v_n} = \eta_v
$$
in the $C^\infty$ topology.

For $a\in A$ and $v\in L$, the function $a$ is bounded on $\supp(\mu)$, so the product $a v$ lies in $L^2(\mu)$. Furthermore, one can check that this product satisfies the usual product rule
\begin{equation}\label{eq:product-rule}
  X[av] = X[a]v + aX[v],
\end{equation}
for any $X\in \mathfrak{g}$, so in fact $av\in L$ and the map $(a,v) \mapsto a v$ is continuous.

Let
\begin{equation}\label{eqn:L20}
L^2_0(\mu):=\left\{f\in L^2(\mu): \int_{Y_0} f\ud \mu = 0\right\},
\end{equation}
and let $L_0 := L\cap L^2_0(\mu)$. Since $\supp \mu$ is compact, every $a\in A$ satisfies $\|a\|_{L^2(\mu)}<\infty$. By Lemma~\ref{lem:link-amenable}, $a \in L_0$ if and only if $\overline{a_\psi} = 0$. We equip $L_0$ with the topology it inherits as a subspace of $L$.

Our main goal in this section is to use results of van Est and Blanc to prove the following lemma.
\begin{lemma}\label{lem:vanishing2}
  For any $n\ge 0$,
  $$\overline{H}^n(\mathfrak{g};L_0) = 0.$$
  Thus, with $\psi$ and $\mu$ as above, if $\beta\in C^n(\mathfrak{g};A)$ is a cocycle and $\overline{\beta_\psi} = 0$, then there is a sequence $\eta_i \in C^{n-1}(\mathfrak{g};L_0)$, $i\in \N$, such that $\lim_i \mathrm{d}\eta_i = \beta$. That is, for any $\lambda\in \wedge^n\mathfrak{g}$, $\mathrm{d}\eta_i(\lambda)$ converges in $L$ to $\beta(\lambda)$.
\end{lemma}

We first recall some definitions. {We will say that a $G$--module $E$ is \textit{complete} if $E$ is a Fréchet space.}
For any {complete} $G$--module {$E$} , and any $n\ge 0$, let
\begin{equation}\label{eqn:Fn}
  F^n_s(G;{E}):=\{f:G^{n+1}\to {E}:\text{$f$ is smooth}\}.
\end{equation}
{Here we recall that when we say that $f$ \textit{is smooth} we mean that all Fréchet derivatives of $f$ exist, and they are continuous}. Let $\mathrm{d}=\mathrm{d}_n\from F^n_s(G;{E})\to F^{n+1}_s(G;{E})$ be the coboundary operator
$$
\mathrm{d}f(g_0,\dots, g_{n+1}) = \sum_{i=0}^{n+1} (-1)^i f(g_0,\dots,\hat{g}_i,\dots, g_{n+1}),
$$
for $f\in F^n_s(G;{E})$. 
Let $C_s^n(G;{E})$ be the set of $G$--equivariant elements of $F^n_s(G;{E})$, i.e., those that satisfy 
\[
    f(h g_0, \ldots, hg_n):= h \cdot\left(f(g_0,\ldots,g_n)\right), \quad \forall h,g_0,\ldots,g_n \in G.
\]
We equip $C^n_s(G;{E})$ with the topology of $C^\infty$ convergence on compact subsets of $G^{n+1}$, and we define the \emph{smooth cohomology} $H^*_s(G;{E})$ of $G$ to be the cohomology of $C_s^*(G;{E})$.

Blanc {showed that the reduced cohomology $\overline{H}_s^*(G;E)$ vanishes when $E$ is a unitary representation.}
\begin{thm}[\cite{Blanc}]\label{thm:blanc}
  Let $G$ be a connected nilpotent Lie group and 
  let $U$ be a unitary representation of $G$ with no nontrivial fixed points, i.e., if $v\in U$ and $g\cdot v = v$ for all $g$, then $v=0$. Let
  $$\overline{H}^k_s(G;U) = \ker \mathrm{d}_k /\clos(\im \mathrm{d}_{k-1}).$$  
  Then $\overline{H}^*_s(G;U) = 0$.
\end{thm}

Blanc proved this theorem for $\overline{H}^*_c(G;U)$, the reduced continuous cohomology of $G$, but by a convolution argument (see \cite[Equation (1.16) and Proposition 1.7(i), page 183]{Guichardet}), there is a topological isomorphism between $\overline{H}^*_c(G;U)$ and $\overline{H}^*_s(G;U)$. The analogous statement when $G$ is a discrete group is Theorem~4.1.3 of \cite{Shalom}.

{We will say that a complete $G$-module {$E$} is $C^\infty$ if the map $\eta_v(g)=g\cdot v$ is smooth for every $v\in {E}$. In this case, ${E}$ can be equipped with the $\mathfrak{g}$--module structure given by \eqref{eq:module-structures}, cf. \cite[Section D.4.1]{Guichardet}. For instance, we can make $L_0$ (and analogously $L$) a complete $C^\infty$-module by equipping it with the $C^\infty$ topology induced by the inclusion $L_0\subset C^\infty(G,L^2_0(\mu))$, $v\mapsto \eta_v$ cf. \cite[Page 347]{Guichardet}. This topology makes $L_0$ into a Fr\'echet space.} {The basic terminology and exposition about $G$-modules, which we are using in this paper, is in Appendices D.3, D.4, E.1, and E.4 in \cite{Guichardet}.} Van Est gave an isomorphism between $H_s^*(G;{E})$ and $H^*(\mathfrak{g};{E})$. (See \cite[Section 7.3]{Guichardet} for additional details on Theorem \ref{thm:vanest}.)
\begin{thm}[{\cite[Theorem 14.1]{vanEstResultPrecise}, \cite[Corollaire 7.2]{Guichardet}}]\label{thm:vanest}
  {Let ${E}$ be a complete $C^\infty$-module}. For all $n$, let $R\from C_s^*(G;{E}) \to C^*(\mathfrak{g};{E})$ be the chain map such that for $f\in C_s^n(G;{E})$,
  $$
  Rf(X_1,\dots, X_n) := \sum_{\sigma\in S_n} \operatorname{sgn}(\sigma)X_{\sigma_1,2}\dots X_{\sigma_n,n+1}f(\id_G,\dots,\id_G),
  $$
  where $X_{i,j}f$ denotes the derivative along $X_i$ in the $j$th argument of $f$.
  
  Then $R$ induces an isomorphism from $H_s^*(G;{E})$ to $H^*(\mathfrak{g};{E})$.
\end{thm}

We use these results to prove Lemma~\ref{lem:vanishing2}.
\begin{proof}[Proof of Lemma~\ref{lem:vanishing2}]
In the following, we equip $L^2(\mu)$ and $L^2_0(\mu)$ with the topology induced by the $L^2$ norm, and we equip $L_0$ with the $C^\infty$ topology induced by the inclusion $L_0\subset C^\infty(G,L^2_0(\mu))$, $v\mapsto \eta_v$. {When $V$ is a Hilbert space}, we view $C^d_s(G;V)$ as a space of functions from $G^{d+1}$ to $V$ and equip it with the $C^\infty$ topology.

For $k$ a compactly supported smooth function on $G$, $V$ a Hilbert space, and $f\from G\to V$, let
$$
(f\ast k)(x) := \int_G f(x g) k(g^{-1})\ud g,
$$
for any $x\in G$. For $v\in L^2_0(\mu)$, let 
$v\ast k \in L^2_0(\mu)$,
$$v\ast k := \int_G (g\cdot v) k(g^{-1}) \ud g,$$
so that for $x\in G$,
$$\eta_{v\ast k}(x) = x \cdot \int_G (g\cdot v) k(g^{-1}) \ud g = \int_G (x g\cdot v) k(g^{-1})\ud g = (\eta_v \ast k)(x).$$
Then for $h\in G$ and $X\in \mathfrak{g}$,
\begin{align*}
X \eta_{v\ast k}(h) &= \frac{\mathrm{d}}{\mathrm{d} t}\bigg|_{t=0} \int_G (h \exp(tX) g\cdot v) k(g^{-1}) \ud g \\
&= \frac{\mathrm{d}}{\mathrm{d} t}\bigg|_{t=0} \int_G (h g\cdot v) k(g^{-1}\exp(tX)) \ud g =  \eta_{v\ast Xk}(h).
\end{align*}
Since $k$ is smooth, $v\ast k\in L_0$. Furthermore, if $v_i, w\in L^2_0(\mu)$, and $\|v_i-w\|_{L^2(\mu)}\to 0$, then $\eta_{v_i \ast k}$ converges to $\eta_{w \ast k}$ in the $C^{\infty}$ topology. That is, for any $k$, the map $v\mapsto v\ast k$ is a continuous linear map from $L^2_0(\mu)$ to $L_0$.

For $k$ as above and $\sigma\in C^d_s(G;L^2_0(\mu))$, we let $\sigma\ast k\in C^d_s(G;L_0)$,
  $$
  (\sigma\ast k)(g_0,\dots,g_d) := \sigma(g_0,\dots,g_d)\ast k,
  $$
  for all $g_0,\dots, g_d\in G$. This satisfies
  $$\mathrm{d}[{\sigma}\ast k] = \mathrm{d}[{\sigma}]\ast k,$$
  and $\sigma\mapsto \sigma \ast k$ is a continuous linear map from $C_s^d(G;L^2_0(\mu))$ to $C_s^d(G;L_0)$.

  Now we prove the lemma.
  Let $R\from C_s^*(G;L_0)\to C^*(\mathfrak{g};L_0)$ be the chain map from Theorem~\ref{thm:vanest}. Note that $R$ is continuous.
  Let $S\from C^*(\mathfrak{g};L_0)\to C_s^*(G;L_0)$ be a homotopy inverse of $R$. Let $\beta\in C^d(\mathfrak{g};L_0)$ such that $\mathrm{d}\beta=0$   and let $\gamma\in C^{d-1}(\mathfrak{g};L_0)$ be such that $\mathrm{d} \gamma = \beta - RS\beta$.

  Since $\mu$ is ergodic, $L^2_0(\mu)$ satisfies the assumption of Theorem~\ref{thm:blanc}. Therefore, there are $h_1,h_2,\dots \in C^{d-1}_s(G;L^2_0(\mu))$ such that $\mathrm{d}h_n$ converges to $S\beta$ in $C^{d}_s(G;L^2_0(\mu)).$

  Let $k_1,k_2,\dots\from G\to \R$ be smooth nonnegative functions such that $\supp k_i\subset B_{\frac{1}{i}}(\id_G)$ and $\int_G k_i(x)\ud x=1$ for all $i$, so that $\lim_i c \ast k_i = c$ for any $c \in C^{d}_s(G;L_0)$.

  Then $\mathrm{d}[h_n\ast k_i] = \mathrm{d}[h_n]\ast k_i$, and since the convolution by $k_i$ is continuous,
  $$\lim_n \mathrm{d}[h_n]\ast k_i = S\beta \ast k_i$$
  as elements of $C^{d}_s(G;L_0)$. Since $\lim_i S\beta \ast k_i = S\beta$, we can choose a sequence $n_i\ge 0$ such that
  $$\lim_i \mathrm{d}[h_{n_i}\ast k_i] = S\beta.$$
  Let $\hat{h}_i := h_{n_i}\ast k_i$. By the continuity of $R$,
  $$\lim_{i\to \infty} \ud[\gamma + R \hat{h}_i] = (\beta - RS\beta) + RS\beta = \beta,$$
  so $\overline{H}^d(\mathfrak{g};L_0)=0$.

  Finally, if $\beta\in C^n(\mathfrak{g};A)$ is a cocycle and $\overline{\beta_\psi}=0$, then by Lemma~\ref{lem:link-amenable}, $\int V_\beta \ud \mu = \overline{\beta_\psi}=0$. In particular, for any $\lambda \in \wedge^n\mathfrak{g}$, 
  $$\int \beta(\lambda)(\phi) \ud \mu(\phi) = \int V_\beta(\phi)(\lambda) \ud \mu = \overline{\beta_\psi}(\lambda) = 0,$$
  so $\beta(\lambda)\in L_0$ and thus $\beta\in C^n(\mathfrak{g};L_0)$. Therefore, the reduced cohomology class of $\beta$ is zero, as desired.
\end{proof}

\begin{remark}\label{rem:A0L0}
  Let $A_0 := \{f\in A\mid \int_{Y_0} f\ud \mu = 0\}$. Then Lemma~\ref{lem:vanishing2} shows that any cocycle $\beta\in C^d(\mathfrak{g};A_0)$ vanishes in $\overline{H}^d(\mathfrak{g};L_0)$; it is unclear whether $\beta$ also vanishes in $\overline{H}^d(\mathfrak{g};A_0)$, and the relationship between $A_0$ and $L_0$ is subtle.

  On one hand, we can relate the $L^2(\mu)$ norm to a norm on certain functions on $G$.  For $\gamma\from G\to \C$, we define the \emph{amenable norm} of $\gamma$ by 
  \begin{equation}\label{eq:def-amen-norm}
    \|\gamma\|_M := \sqrt{\lim_{n\to \infty} \fint_{B_n} |\gamma(g)|^2 \ud g},
  \end{equation}
  if the limit exists. If $\beta\in A$ and $\psi\in Y_0$ is an ergodic map with limit measure $\mu$, then
  $$\|\beta_\psi\|_M = \sqrt{\lim_{n\to \infty} \fint_{B_n} |\beta(\psi \cdot g)|^2 \ud g} = \sqrt{\int |\beta(y)|^2 \ud \mu(y)} = \|\beta\|_{L^2(\mu)}.$$
  Thus, if $\beta \in A$ and $\|\beta\|_{L^2(\mu)}$ is small, then $\beta_\psi$ has small amenable norm.

  By Lusin's Theorem and Proposition~\ref{prop:PropertiesOfCinftyinfty}.(3), any $f\in L$ is the limit in $L^2(\mu)$ of a sequence of functions $\alpha_i\in A$; when $i$ and $j$ are large, $\|\alpha_i-\alpha_j\|_{L^2(\mu)}$ is small, so $\|(\alpha_i)_\psi-(\alpha_j)_\psi\|_M$ is small.

  On the other hand, while we can approximate any $f\in L$ in the $L^2(\mu)$ norm by functions $\alpha_i\in A$, it is unclear whether we can approximate $f$ in the topology on $L$. In particular, for $X\in \mathfrak{g}$, the derivative operator $h\mapsto X[h]$ from $L$ to $L$ is unbounded in the $L^2(\mu)$ norm, so it is not clear whether we can choose the $\alpha_i$'s such that the derivatives $X[\alpha_i]$ also converge in $L^2(\mu)$. It is possible that $L$ is substantially larger than $A$, but we have not been able to construct explicit elements of $L$ that do not lie in the closure of $A$.
\end{remark}

\section{Amenable averages and cohomology}\label{sec:Integration}

In this section, we will use the maps constructed in the previous sections to prove the first part of Theorem~\ref{thm:induced-map}. We will first show that if $\psi\in Y_0$ is an ergodic map, then the map ${\overline{\psi^*}}\from H^*(\mathfrak{h})\to H^*(\mathfrak{g})$ defined by \eqref{eqn:ContinuousInvariant} induces a homomorphism of cohomology algebras over $\mathbb C$. Since ${\overline{\psi^*}}$ descends to a map ${\overline{\psi^*}}\from H^*(\mathfrak{h};\mathbb R)\to H^*(\mathfrak{g};\mathbb R)$, the first part of Theorem~\ref{thm:induced-map} follows.

We first write ${\overline{\psi^*}}$ in terms of a projection from $L$ to $\C$.
As in the previous sections, let $\psi\in Y_0$ be an ergodic map, let $\mu$ be its limit measure, and let $L=(L^2(\mu))^\infty$ defined in \eqref{eqn:L2infty}. Let $M\from L\to \C$,
$$
M f := \int_{Y_0} f\ud \mu.
$$
Since $\mu$ is $G$--invariant, $M$ is a continuous $G$--module homomorphism with respect to the trivial $G$--action on $\C$. Therefore, $M$ induces maps from $C^*(\mathfrak{g};L)$ to $C^*(\mathfrak{g})$, and from $H^*(\mathfrak{g};L)$ to $H^*(\mathfrak{g})$. We denote these maps by $M$ as well. 

When $f\in A$, Lemma~\ref{lem:link-amenable} lets us write $Mf$ in terms of the amenable average:
$$
Mf = \int_{Y_0} f \ud \mu  = \lim_{n\to\infty}\fint_{B_n}f_\psi\ud g = \overline{f_\psi}.
$$
Likewise, when $\alpha\in C^d(\mathfrak{g};A)$ and $\lambda \in \wedge^d \mathfrak{g}$,
\begin{equation}\label{eqn:DefinitionOfM}
M\alpha(\lambda) =  \int_{Y_0} \alpha(\lambda) \ud \mu = \lim_{n\to\infty}\fint_{B_n}\alpha(\lambda)_\psi\ud g \stackrel{\eqref{eq:def-amen-avg}}{=} \overline{\alpha_\psi}(\lambda).
\end{equation}

\begin{lemma}\label{lemma:LimitExists}
  Let $T^\sharp$ be the map defined in Lemma~\ref{lem:Tsharp}. For every left-invariant form $\omega \in C^d(\mathfrak{h})$, the amenable average 
  $$
  \overline{\psi^*}\omega:=\lim_{n\to \infty} \fint_{B_n} (\psi^*\omega)_g \ud g$$ 
  exists and satisfies $\overline{\psi^*}\omega=(M\circ T^\sharp)(\omega)$.
\end{lemma}
\begin{proof}
  Let $\omega\in C^d(\mathfrak{h})$. Then $M(T^\sharp\omega)$ exists and for any $\lambda\in \wedge^d\mathfrak{g}$,
  \[
    M(T^\sharp\omega)(\lambda) \stackrel{\eqref{eqn:DefinitionOfM}}{=} \lim_{n\to\infty}\fint_{B_n} T^\sharp\omega(\lambda)_\psi\ud g \stackrel{\eqref{eqn:ImportantEquality}, \eqref{eq:def-form-eval}}{=} \lim_{n\to\infty}\fint_{B_n}(\psi^*\omega)_g(\lambda) \ud g,
  \]
  which proves the lemma.
\end{proof}
Let $\overline{\psi^*} := M\circ T^\sharp \from C^*(\mathfrak{h})\to C^*(\mathfrak{g})$. By  Lemma~\ref{lemma:LimitExists}, this satisfies \eqref{eqn:ContinuousInvariant}.
Since $T^\sharp$ is a homomorphism of differential graded algebras, $T^\sharp$ induces a homomorphism of cohomology algebras. We claim that $M$ also induces a homomorphism of cohomology algebras, which would imply that $\overline{\psi^*}$ does as well. For $u\in C^*(\mathfrak{g};A)$, let $[u]\in H^*(\mathfrak{g};A)$ denote the corresponding cohomology class. 

\begin{lemma}\label{lem:McommutesCupProduct}
  The map $M$ induces a homomorphism of cohomology algebras from $H^*(\mathfrak{g};A)$ to $H^*(\mathfrak{g})$. That is, for every $[u_1],[u_2]\in H^*(\mathfrak{g};A)$,
  \begin{equation}
    [Mu_1 \wedge Mu_2] = [M(u_1\wedge u_2)], \qquad \text{in $H^*(\mathfrak{g})$}.
  \end{equation}
\end{lemma}
\begin{proof}
  Let $L:=(L^2(\mu))^{\infty}$, $L_0 :=L^2_0(\mu)\cap L$, and $A_0:= \{f\in A : \int_{Y_0} f \ud \mu = 0 \}$ as in the previous sections. We can decompose $A=A_0\oplus \C$ and $L=L_0\oplus \C$, where $\C$ is the space of constant functions on $Y_0$; then $M\from L\to \C$ is the projection to $\C$.

  For a module $Z$, let $\overline{B}^k (\mathfrak{g};Z) := \clos( \mathrm{d}^{k-1}) \subset C^{k}(\mathfrak{g};Z))$ be the space of reduced coboundaries with coefficients in $Z$. 
  Let $\alpha\in C^{m}(\mathfrak{g};A)$, $\beta\in C^{n}(\mathfrak{g};A)$ be cocycles. We claim that $M(\alpha) \wedge M(\beta)$ and $\alpha\wedge \beta$ lie in the same reduced cohomology class in $\overline{H}^*(\mathfrak{g};L)$, i.e., 
  \begin{equation}\label{eq:congruence-goal}
    M(\alpha) \wedge M(\beta)\equiv \alpha\wedge \beta \pmod{\overline{B}^* (\mathfrak{g};L)}.
  \end{equation}

  Let $\beta_0 := \beta - M\beta$. Then $\beta_0$ is a cocycle in $C^n(\mathfrak{g};A_0)$. By Lemma~\ref{lem:link-amenable}, and Lemma~\ref{lem:vanishing2}, there is a sequence $(\eta_i)_i \in C^{n-1}(\mathfrak{g};L_0)$ such that {$\lim_i \mathrm{d}\eta_i = \beta_0$}.

  Recall that for $a\in A$ and $b\in L$, we can define a continuous product $ab\in L$ that satisfies the usual product rule \eqref{eq:product-rule}. Thus, for $\alpha\in C^{m}(\mathfrak{g};A)$ and $\eta_i\in C^{n-1}(\mathfrak{g};L_0)$, we can define the wedge product $\alpha\wedge \eta_i$ as in \eqref{eqn:DefnCupProduct}. This is continuous, so 
  $$\lim_{i\to\infty} \mathrm{d}[\alpha\wedge \eta_i] = \lim_{i\to\infty} \mathrm{d}[\alpha]\wedge \eta_i + (-1)^{m} \alpha\wedge \mathrm{d}[\eta_i] = (-1)^{m} \alpha\wedge \beta_0.$$
  That is, $\alpha\wedge (\beta-M\beta) \in \overline{B}^* (\mathfrak{g};L)$, so $\alpha\wedge \beta \equiv \alpha \wedge M(\beta) \pmod{\overline{B}^* (\mathfrak{g};L)}.$

  Applying this argument to $M(\beta)$ and $\alpha$, we find $\alpha\wedge M(\beta) \equiv M(\alpha) \wedge M(\beta)$, so
  \begin{equation}\label{eq:reduced-cohomologous}
    \alpha\wedge \beta \equiv \alpha \wedge M(\beta) \equiv M(\alpha) \wedge M(\beta) \pmod{\overline{B}^* (\mathfrak{g};L)}.
  \end{equation}
  We have $M(\overline{B}^* (\mathfrak{g};L)) \subset \overline{B}^* (\mathfrak{g}) = B^*(\mathfrak{g})$, where the last equality follows from the finite-dimensionality of $C^* (\mathfrak{g})$. Therefore, we can apply $M$ to both sides of \eqref{eq:reduced-cohomologous} to conclude
  $$M(\alpha\wedge \beta) \equiv M(M(\alpha) \wedge M(\beta)) = M(\alpha) \wedge M(\beta) \pmod{B^*(\mathfrak{g})},$$
  as desired.
\end{proof}
This proves the first part of Theorem \ref{thm:induced-map}. 

\section{Quasi-isometries and cohomology}\label{sec:proof}

Finally, we prove that if $\psi\in Y_0$ is an ergodic quasi-isometry, then $\overline{\psi^*}$ is an isomorphism of cohomology algebras. We first show that $\overline{\psi^*}$ preserves the top-dimensional cohomology.

We will need the following lemma. Recall that, for $C> 0$, and $L\geq 1$, $\psi:G\to H$ is an \textit{$(L,C)$--quasi-isometry} if
\begin{equation}\label{eqn:LCQiEmbedding}
L^{-1}d_G(g_1,g_2)-C \leq d_H(\psi(g_1),\psi(g_2))\leq Ld_G(g_1,g_2)+C, \quad \forall g_1,g_2\in G,
\end{equation}
and if for all $h\in H$, there is a $g\in G$ such that $d_H(\psi(g),h)\leq C$. We say that $\psi$ is an \emph{$(L,C)$-quasi-isometric embedding} if we only have \eqref{eqn:LCQiEmbedding}.
\begin{lemma}\label{lem:VolumeControl}
  Let $G$ and $H$ be two simply connected nilpotent Lie groups. Fix orientations on $G$ and $H$ and let $\psi\from G\to H$ be a Lipschitz $(L,C)$--quasi-isometry. Let $\theta$ be the degree of $\psi$, let $\omega$ be the volume form on $H$, and let
  $$
  \tau(R) := \int_{B_R^G} \psi^*\omega \ud g.
  $$
  Then $\theta \in \{-1,1\}$, and there is a $c$ depending on $G$, $H$, $L$, and $C$ such that $\theta \tau(R) \ge c |B_R^G|$ for all sufficiently large $R$.
\end{lemma}

Given this lemma, we can show that  $\overline{\psi^*}([\omega])\ne [0]$.
\begin{cor}\label{cor:volcontrol}
  Let $\psi\in Y_0$ be an ergodic quasi-isometry, and $\omega$ be the volume form on $H$. Let $n$ be the topological dimension of $H$. Then $[\omega]$ is a generator of $H^n(\mathfrak{h};\mathbb R)$, and
  \begin{equation}\label{eqn:WantedEquation}
    \overline{\psi^*}([\omega])\neq [0].
  \end{equation}
\end{cor}
\begin{proof}
  Let $\lambda \in \wedge^n \mathfrak{g}$ be the positively oriented unit $n$--vector. Assume $\theta=1$ in Lemma \ref{lem:VolumeControl}, the other case is analogous. Then
  $$|\overline{\psi^*}(\omega)(\lambda)| = \lim_{R\to \infty} |B^G_R|^{-1} \tau(R)\ge c,$$
  so $\overline{\psi^*}(\omega)\ne 0$, and \eqref{eqn:WantedEquation} follows because $G$ is $n$-dimensional too.
\end{proof}

Before proving Lemma~\ref{lem:VolumeControl}, we recall some facts about the degree of a map and the area formula {which can for instance be found in} \cite[Section IV]{MappingDegreeTheory}, and \cite{Federer1996GeometricTheory}, respectively. We fix two simply connected nilpotent Lie groups $G$ and $H$ of the same topological dimension, and equip them with Riemannian metrics and orientations. We recall that a map $f\from X\to Y$ between topological spaces is \textit{proper} if $f^{-1}(y)$ is compact for every $y\in Y$.

Let $U\subset G$ be an open set with closure $\overline{U}$, let $f\from \overline{U}\subset G\to H$ be a continuous function, and $h\in H$. We say that $(f,U,h)$ is \textit{admissible} if $f$ is proper, and $h\not \in f(\partial U)$. If in addition $f$ is smooth, and $h$ is a regular value for $f$, we define
\begin{equation}\label{eqn:DegDef}
\deg(f,U,h):=\sum_{x\in f^{-1}(h)\cap U} \sgn(Df_x),
\end{equation}
where $\sgn(Df_x) = 1$ if $Df_x$ is orientation-preserving and $-1$ if $Df_x$ is orientation-reversing. Since $f$ is proper and $h$ is a regular value, $f^{-1}(h)$ is a finite set and this is a finite sum.

When $(f,U,h)$ is admissible, but $f$ is not smooth or $h$ is not a regular value, we can still define $\deg(f,U,h)$ by approximating by smooth functions. We let
\[
\deg(f,U,h):=\deg(\tilde f,U,h'),
\]
where $(\tilde f,U,h')$ is admissible, $\tilde f\from \overline{U}\to H$ is smooth, $h'$ is a regular value for $\tilde f$, and 
\[
d_H(h,h') + \sup_{g\in \overline{U}}d_H(f(g),\tilde f(g))\ll \mathrm{dist}_H(h,f(\partial U)).
\]
(In particular, when $U=G$, we require that $\sup_{g\in G} d_H(f(g),\tilde f(g)) < \infty$.)

It is well-known that this definition is independent of the choice of $\tilde f$ and $h'$. In fact, one can show the following, see \cite[Proposition 2.3, page 146 \& Proposition 2.4, page 147 \& Corollary 2.5(3), page 149]{MappingDegreeTheory}.
\begin{prop}\label{prop:degree-props}
  Let $U\subset G$ be an open set and let $f\from \overline{U} \to H$ be a proper map.
  \begin{enumerate}
  \item (Basepoint invariance) If $h$ and $h'$ lie in the same path component of $H\setminus f(\partial(U))$, then $\deg(f,U,h) = \deg(f,U,h')$. Further, if $f\from G\to H$ is proper, then $\deg(f,G,h)$ is independent of $h$, and we define $\deg(f):=\deg(f,G,h)$.
  \item (Homotopy invariance) If $f_t\from \overline{U}\to H, t\in [0,1]$ is a homotopy such that $f_t$ is proper and $h\not \in f_t(\partial U)$ for all $t$, then $\deg(f_0,U,h) = \deg(f_1,U,h)$.
  \item (Excision) If $D\subset U$ is an open set and $f^{-1}(h)\subset D$, then $\deg(f,U,h) = \deg(f,D,h)$.
  \end{enumerate}
\end{prop}

We also recall the area formula, see, e.g., \cite[Theorem 3.2.3]{Federer1996GeometricTheory}, and \cite[Corollary 4.1.26]{Federer1996GeometricTheory}. Let $U\subset G$ be an open set, and let $f\from \overline{U}\to H$ be a proper Lipschitz map. Then, denoting with $\omega$ the volume form on $H$, we have 
\begin{equation}\label{eq:signed-area}
\int_U f^*\omega = \int_U \det(Df(g))\ud g = \int_{f(U)}\deg(f,U,h)\ud h,
\end{equation}
and
\begin{equation}\label{eq:unsigned-area}
  \int_U |\det(Df(g))| \ud g = \int_{H} \#f^{-1}(h) \ud h \ge \int_{f(U)}|\deg(f,U,h)|\ud h.
\end{equation}

{From now on, in this last section, the identity element of the group $G$ will be denoted with $0$}. Finally, recall that if $G$ is a simply connected nilpotent Lie group, then the rescalings $(G,r^{-1}d_G)$ Gromov--Hausdorff converge, as $r\to\infty$, to a scaling limit $G^\infty$, which is the associated graded nilpotent Lie group equipped with a subRiemannian metric. Furthermore, there is a family of homeomorphisms $s^G_r\from G^\infty\to G, r>0$ with $s^G_r(0)=0$ such that for all $v,v'\in G^{\infty}$, 
$$\lim_{r\to \infty} r^{-1} d_G(s_r(v),s_r(v')) = d_{G^\infty}(v,v'),$$
locally uniformly in $v$ and $v'$, and such that for all $R>0$,
$$
\lim_{r\to \infty} r^{-1}d_{\text{Haus}}(s^G_r(B_R^{G^\infty}), B_{rR}^G) = 0.
$$
In fact, we can write $s^G_r = \exp_G \circ M_r \circ \exp_{G^{\infty}}^{-1}$, where $M_r$ is a family of linear maps from $\mathfrak{g}^\infty$ to $\mathfrak{g}$, see the scheme of the proof in \cite[Theorem 4.7]{LeDonnePrimer}, and the original proof in \cite[Section (38)--(49)]{PansuCroissance}, in particular \cite[Proposition (47)]{PansuCroissance}. 

If $\psi\from G\to H$ is a quasi-isometry such that $\psi(0)=0$ and $s^H_r\from H^\infty\to H, r>0$ is the corresponding family for $H$, then we can define a family of maps $\psi_r\from G^\infty\to H^\infty$,
\begin{equation}\label{eq:psi-rescale}
  \psi_r := (s_r^H)^{-1} \circ \psi \circ s_r^G.
\end{equation}
Then there is a subsequence $(r_i)_i$, with $r_i\to\infty$, such that $\psi_{r_i}$ converges locally uniformly to a bilipschitz homeomorphism $\psi_\infty\from G^\infty \to H^\infty$ as $i\to \infty$.

For $U\subset G$ or $U\subset H$, let $|U|$ be the Riemannian volume of $U$. For $\Lambda = G, H$ and $R>0$, let $B_R^\Lambda\subset \R^n$ be the {ball of radius $R$} centered at $0$ with respect to $d_\Lambda$.

\begin{proof}[Proof of Lemma~\ref{lem:VolumeControl}]
  Without loss of generality, suppose that $\psi(0)=0$.
  For $r>0$, let $\psi_r \from G^\infty \to H^\infty$ as in \eqref{eq:psi-rescale}, and let $(r_i)_i$ be a sequence such that $\psi_{r_i}$ converges locally uniformly to $\psi_\infty\from G^\infty\to H^\infty$, as $r_i\to\infty$.

  Suppose that $\psi_\infty$ is positively oriented, i.e., $\deg(\psi_\infty)=1$; the negatively oriented case is similar. Let $B=B^{G^\infty}_1$ be the unit ball in $G^\infty$. Since $\psi$ is a quasi-isometry, it is proper, and $\psi^{-1}(0)$ is compact. Thus, when $i$ is sufficiently large, $s^G_{r_i}(B)\Supset \psi^{-1}(0)$, and by excision, $\deg(\psi, G, 0) = \deg(\psi, s^G_{r_i}(B), 0)$. Therefore,
  $$
    \deg(\psi) = \lim_{i\to \infty} \deg(\psi, s^G_{r_i}(B), 0)   =\lim_{i\to \infty} \deg(\psi_{r_i}, B, (s_{r_i}^H)^{-1}(0)) 
    = \lim_{i\to \infty} \deg(\psi_{r_i}, B, 0).
  $$
  Since  $\psi_{r_i}$ converges uniformly to $\psi_{\infty}$ on $B$, the homotopy invariance of the degree implies that 
  $$\lim_{i\to \infty} \deg(\psi_{r_i}, B, 0) = \deg(\psi_{\infty}, B, 0) = 1,
  $$
  and thus $\deg(\psi) = \deg(\psi, G, 0) = 1$.
  
  Let $R>4CL$ and let $R' = R - 2CL$. If $x\in B_{R'}^G$ and $y\not\in B_R^G$, then $d_G(x,y)\ge 2CL$, so, since $\psi$ is an $(L,C)$-quasi-isometry, $d_H(\psi(x),\psi(y))>0$, i.e., $\psi(x)\ne \psi(y)$. Let $Z=\psi(B_R^G)$ and $Z'=\psi(B_{R'}^G)$. Then $\psi^{-1}(Z')\subset B_R^G$, so by excision, for all $z\in Z'$, 
  \begin{equation}\label{eq:z-degree}
    \deg(\psi, B_R^G, z) = \deg(\psi,G, z) = 1.
  \end{equation}
  Now we calculate $\tau(R)=\int_{B_R^G} \psi^*\omega \ud g$. By the area formula \eqref{eq:signed-area}, and \eqref{eq:z-degree},
  $$\tau(R) = \int_{Z} \deg(\psi, B_R^G, z) \ud z \stackrel{\eqref{eq:z-degree}}{=} \int_{Z\setminus Z'} \deg(\psi, B_R^G, z) \ud z + |Z'|.$$
  Then
  $$
  \big|\tau(R) - |Z'|\big| \le \int_{Z\setminus Z'} |\deg(\psi, B_R^G, z)| \ud z \le \int_{Z\setminus Z'} \#\left(B_R^G\cap \psi^{-1}(z)\right) \ud z.
  $$
  Let $M := B_R^G\setminus B_{R'}^G$. Then $B_R^G\cap \psi^{-1}(z) \subset M$ for every $z\in Z\setminus Z'$, so, using \eqref{eq:unsigned-area}, 
  $$
  \big|\tau(R) - |Z'|\big| \le \int_{H} \#\left((\psi|_M)^{-1}(z)\right) \ud z = \int_{M} |\det(Df(g))| \ud x \le |M| \Lip(\psi)^n.
  $$
  If $d_G(0,g)>R'$, then $d_H(0,\psi(g)) > L^{-1} R' - C \ge\frac{R}{4L}$, so $\psi^{-1}(B_{\frac{R}{4L}}^H)\subset B_{R'}^G$. Since $\psi$ is {of degree $1$}, it is surjective, so this implies $B_{\frac{R}{4L}}^H\subset \psi(B_{R'}^G)$.
  Thus, 
  $$
  |Z'| = \left|\psi(B_{R'}^G)\right| \ge \left|B_{\frac{R}{4L}}^H\right|.
  $$
  Since $G$ and $H$ are quasi-isometric, they have the same growth exponent $Q$, and there is a $c>0$ depending only on $G, H, L$, and $C$ such that when $R$ is large, $|Z'| \ge c R^Q$. Likewise, there is a $c'>0$ such that when $R$ is large,
  $$
  |M| \Lip(\psi)^n = \Lip(\psi)^n(\left|B_{R}^G\right| - \left|B_{R'}^G\right|) \le c' R^{Q-1}.
  $$
  Thus, when $R$ is sufficiently large,
  $$\tau(R) \ge |Z'| - |M| \Lip(\psi)^n \ge c R^Q - c' R^{Q-1} \ge \frac{c}{2} R^Q,$$
  as desired.
\end{proof}

Finally, we complete the proofs of Theorem~\ref{thm:mainThm} and Theorem \ref{thm:induced-map}.
\begin{proof}[Proof of Theorem \ref{thm:induced-map}]
As proved in Section \ref{sec:Integration}, a direct application of Lemma \ref{lemma:LimitExists}, Lemma \ref{lem:Tsharp}, and Lemma \ref{lem:McommutesCupProduct}, gives that the map $\overline{\psi^*}$ in \eqref{eqn:ContinuousInvariant} exists and induces a homomorphism $\overline{\psi^*}:H^*(\mathfrak{h};\mathbb R)\to H^*(\mathfrak{g};\mathbb R)$ of cohomology algebras over $\mathbb R$.

If $\psi$ is a quasi-isometry, then by Corollary \ref{cor:volcontrol} we have $\overline{\psi^*}([\omega])\neq [0]$, where $[\omega]$ is a generator of the top-dimensional cohomology group. Hence, by Poincaré duality, $\overline{\psi^*}$ is injective. Switching the roles of $\mathfrak{g}$ and $\mathfrak{h}$, there is also an injective algebra homomorphism $H^*(\mathfrak{g};\mathbb R)\to H^*(\mathfrak{h};\mathbb R)$. Hence, $H^*(\mathfrak{g};\mathbb R)$ and $H^*(\mathfrak{h};\mathbb R)$ have the same (finite) dimension, so $\overline{\psi^*}$ is an isomorphism of cohomology algebras over $\mathbb R$, as desired.
\end{proof}

{Before proving Theorem~\ref{thm:mainThm} we recall the center-of-mass convolution studied in \cite{KMX1}. A similar construction was performed by Buser--Karcher in \cite{BK81}. Let us recall some definitions. 

Let $G$ be a step-$m$ simply connected nilpotent Lie group endowed with a left-invariant Riemannian metric which induces a distance $d$. Let $|\cdot|_e$ be the norm induced by the scalar product in the Lie algebra $\mathfrak{g}$. Let $\mathfrak{g}_i$ be the lower central series of $\mathfrak{g}$: i.e., $\mathfrak{g}_0:=\mathfrak{g}$, and $\mathfrak{g}_{i+1}:=[\mathfrak{g},\mathfrak{g}_i]$ for all $i\geq 0$. Notice that $\mathfrak{g}_j=\{0\}$ for all $j\geq m$. For $i\geq 0$, let $W_i$ be vector subspaces of $\mathfrak{g}$ such that $\mathfrak{g}_i=W_i\oplus \mathfrak{g}_{i+1}$. We have $\mathfrak{g}=\oplus_{i=0}^{m-1}W_i$, and let us define $\pi_i$ as the projection from $\mathfrak{g}$ to $W_i$. Let us denote
\[
\|v\|:=\sum_{i=0}^{m-1}|\pi_j(v)|_e^{1/(j+1)}, \quad \forall v\in\mathfrak{g}.
\]

Let $\log:G\to\mathfrak{g}$ be the inverse of the exponential map. Notice that for $v\in \log(B_R(0))$ we have
\begin{equation}\label{eqn:Estimates}
\|v\|^m\lesssim |v|_e\lesssim \|v\|, \qquad |v|_e \lesssim d(0,\exp(v))\lesssim |v|_e,
\end{equation}
where the constants in the estimates only depend on the Riemannian group $G$ and $R$. We will say that a Borel probability measure $\xi$ on $G$ has \textit{finite $p$-th moment}, with $p\in\mathbb N$, if 
\[
\int_{\mathfrak{g}}\|v\|^p\mathrm{d}(\log_*(\xi))(v) < +\infty.
\]
\begin{theorem}[{\cite[Theorem 3.4, Remark 3.18]{KMX1}}]\label{thm:COM}
    Let $G$ be a step-$m$ simply connected nilpotent Lie group. Let $\nu$ be a Borel probability measure on $G$ with finite $m$-th moment.
    
    Then, for every $x\in G$ the map $\log_x:=\log\circ l_{x^{-1}}$ defined from $G$ to $\mathfrak{g}$ is $\nu$-integrable. Moreover $C_\nu:G\to\mathfrak{g}$ defined as
    \[
    C_\nu(x):=\int_G \log_x\mathrm{d}\nu,
    \]
    is a diffeomorphism. In addition, there exists a $\mathfrak{g}$-valued polynomial $Q$ with $Q(0,\ldots,0)=0$, and $\mathfrak{g}$-valued multilinear forms $L_1,\ldots,L_K$ such that 
    \begin{equation}\label{eqn:DefnCnu}
    \log\left(C_\nu^{-1}(0)\right)=Q(A_1,\ldots,A_K),
    \end{equation}
    where 
    \begin{equation}\label{eqn:Ai}
    A_i:=\int_{\mathfrak{g}}L_i(v,\ldots,v)\mathrm{d}(\log_*\nu)(v), \quad \forall 1\leq i\leq K,
    \end{equation}
    and 
    \begin{equation}\label{eqn:EstimateLi}
    |L_i(v,\ldots,v)|_e\lesssim 1+\|v\|^m, \quad \forall 1\leq i \leq K.
    \end{equation}
\end{theorem}
\begin{defn}[Center of mass]
    Let $G$ be a step-$m$ simply connected nilpotent Lie group, and let $\nu$ be a Borel probability measure on $G$ with finite $m$-th moment. By using Theorem \ref{thm:COM}, we define the \textit{center of mass} of $\nu$ as follows:
    \[
    \mathrm{com}(\nu):=(C_\nu)^{-1}(0).
    \]
\end{defn}
Let $G$ be a step-$m$ simply connected nilpotent Lie group, and let $\nu$ be a Borel probability measure on $G$ with finite $m$-th moment. Let us denote $I:G\to G$ the inversion map defined by $I(x):=x^{-1}$ for every $x\in G$. By using \cite[Lemma 3.21, Remark 3.26]{KMX1}, we recall that 
\begin{equation}\label{eqn:InvarianceCOM}
\mathrm{com}(I_*\nu)=I(\mathrm{com}(\nu)), \qquad \mathrm{com}((l_x)_*\nu)=l_x(\mathrm{com}(\nu)),\quad \forall x\in G.
\end{equation}
{By a compactness argument applied to Theorem~\ref{thm:COM}, for every $r>0$ there is an $R>0$ (depending on $G$) such that if $\nu$ is supported on the ball $B^G_r(0)$, then $\mathrm{com}(\nu) \in B^G_R(0)$. The second part of \eqref{eqn:InvarianceCOM} then implies that for any $x\in G$, if $\nu$ is supported on $B^G_r(x)$, then $\mathrm{com}(\nu) \in B^G_R(x)$.}

We are now ready to define the mollification. In order not to complicate our discussion, we will stick to mollifying continuous functions, even though, with the same procedure, it is possible to mollify more general functions.
\begin{defn}[Mollification]\label{def:Mollification}
    Let $G,H$ be simply connected nilpotent Lie groups endowed with left-invariant Riemannian metrics. Let $\sigma$ be a probability measure on $G$ such that: it has smooth density with respect to a Haar measure $\mu$ on $G$; it is supported on the ball of center $0$ and radius $1$; it satisfies $I_*\sigma=\sigma$, where $I$ is the inversion map defined above.

    Let $f:G\to H$ be a continuous function. Then, by \eqref{eqn:Estimates} and since $f$ is continuous, it is straightforward to check that for every $x\in G$, $(f\circ l_x)_*\sigma$ has finite $m$-th moment. Thus, by using Theorem \ref{thm:COM}, we can define the \textit{mollification} $f_1:G\to H$ as follows:
    \[
    f_1(x):=\mathrm{com}\left((f\circ l_x)_*\sigma\right), \quad \forall x\in G.
    \]
\end{defn}
Notice that by \eqref{eqn:InvarianceCOM} (arguing verbatim as in \cite[Lemma 3.33(3)]{KMX1}) we have that if $f:G\to H$ is a continuous map and $x\in G$, $y\in H$, then
\begin{equation}\label{eqn:InvarianceMollification}
(l_y\circ f\circ l_x)_1=l_y\circ f_1\circ l_x.
\end{equation}
We are now ready to prove Theorem 1.3.
}

\begin{proof}[Proof of Theorem~\ref{thm:mainThm}]
  Suppose that $G$ and $H$ are quasi-isometric simply connected nilpotent Lie groups. By \cite{BoissonatDyerGhosh}, $G$ admits a triangulation with bounded geometry, so if $f\from G\to H$ is a quasi-isometry, there is a continuous quasi-isometry $f'\from G\to H$ which agrees with $f$ on the vertices of the triangulation. (One can also construct $f$ using the nerve of a collection of points in $G$, see for instance Chapter~3.8 of \cite{DrutuKapovich}.)

  {Let $\phi:G\to H$ be the mollification of $f'$ with respect to a measure $\sigma$ supported in the unit ball $B_1^G$, as in Definition \ref{def:Mollification}, i.e., $\phi(x):= \mathrm{com}((f'\circ l_x)_*\sigma)$. We claim that $\phi\in Y_0$ is a quasi-isometry, i.e., $\phi$ is a quasi-isometry with bounded derivatives. 

  We first show that there is a $C>0$ such that $d(\phi(x),f'(x))\leq C$ for every $x\in G$. Since
$f'$ is a quasi-isometry, there is an $r>0$ such that 
  $$\supp((f'\circ l_x)_*\sigma)\subset f'(B^G_1(x))\subset B^H_r(f'(x))$$
  for all $x$. Therefore, by the remark after \eqref{eqn:InvarianceCOM}, there is a $C>0$ such that $\phi(x)\in B^H_C(f'(x))$ for all $x\in G$. Since $f'$ is a quasi-isometry, $\phi$ is a quasi-isometry as well.
  
  We are left to show that $\phi$ has bounded derivatives. By \eqref{eqn:InvarianceMollification} it is enough to assume $\phi(0)=0$, and to show that the map $\widetilde\phi:=\log_H\circ \phi\circ \exp_G$ is smooth and has uniformly bounded derivatives at $0$. As in the proof of \cite[Lemma 3.33(7)]{KMX1}, Theorem~\ref{thm:COM} implies that $\widetilde{\phi}(x) = \log_H \mathrm{com}((f' \circ l_{\exp_G(x)})_*(\sigma))$ can be written as a smooth function of the integrals
  $$\widetilde{A}_i(x) = \int_{\mathfrak{g}} L_i(v,\dots,v) \ud(\log \circ f' \circ l_{\exp_G(x)})_*(\sigma)$$
  for $i=1,\dots, K$, where each $L_i$ is a multilinear form.
  Since $\mu$ is compactly supported, these integrals are smooth functions of $x$. Since $\supp (f')_*(\sigma) \subset B_c^H(f'(0))\subset B_{c+C}^H(\phi(0))$, when $x$ is small, each $\widetilde{A}_i$ and its derivatives are uniformly bounded by constants only depending on $G,H,k$, the smoothing kernel $\sigma$, and the radii $c$ and $C$ above. Therefore, the derivatives of $\widetilde{\phi}$ are likewise bounded by constants only depending on $G,H,k,\sigma, c$, and $C$, as desired.}

  Finally, by Proposition \ref{prop:ExistenceOfErgodic} there exists an ergodic quasi-isometry $\psi\in Y_0$. By Theorem \ref{thm:induced-map}, the pullback $\overline{\psi^*}$ induces an isomorphism of cohomology algebras $H^*(\mathfrak{h};\mathbb R)\cong H^*(\mathfrak{g};\mathbb R)$ as desired.
\end{proof}

\printbibliography[title={References}]

\end{document}